\documentclass[12pt]{amsart}
\usepackage{amsfonts,latexsym,amsthm,amssymb,graphicx}
\usepackage[all]{xy}
\DeclareFontFamily{OT1}{rsfs}{}
\DeclareFontShape{OT1}{rsfs}{n}{it}{<-> rsfs10}{}
\DeclareMathAlphabet{\mathscr}{OT1}{rsfs}{n}{it}

\CompileMatrices

\newtheorem{theorem}{Theorem}[section]
\newtheorem{lemma}[theorem]{Lemma}
\newtheorem{corol}[theorem]{Corollary}
\newtheorem{prop}[theorem]{Proposition}
\newtheorem{claim}[theorem]{Claim}

{\theoremstyle{definition} }
{\theoremstyle{remark} \newtheorem{remark}[theorem]{Remark}
\newtheorem{example}[theorem]{Example}}

\newcommand{\Abb}{{\mathbb{A}}}
\newcommand{\Cbb}{{\mathbb{C}}}
\newcommand{\Pbb}{{\mathbb{P}}}
\newcommand{\Tbb}{{\mathbb{T}}}
\newcommand{\Ubb}{{\mathbb{U}}}
\newcommand{\Zbb}{{\mathbb{Z}}}

\newcommand{\qede}{\hfill$\lrcorner$}

\newcommand{\cI}{{\mathcal I}}
\newcommand{\cJ}{{\mathcal J}}
\newcommand{\cO}{{\mathscr O}}

\newcommand{\Til}[1]{{\widetilde{#1}}}
\newcommand{\one}{1\hskip-3.5pt1}
\newcommand{\csm}{{c_{\text{SM}}}}

\newcommand{\cone}[1]{{#1}^\wedge}

\title[Chern classes of graph hypersurfaces]{
Chern classes of graph hypersurfaces and deletion-contraction 
relations}
\author{Paolo Aluffi}
\address{
Mathematics Department, 
Florida State University,
Tallahassee FL 32306, U.S.A.
}
\email{aluffi@math.fsu.edu}

\begin{document}

\begin{abstract}
We study the behavior of the Chern classes of graph hypersurfaces under 
the operation of deletion-contraction of an edge of the corresponding graph.
We obtain an explicit formula when the edge satisfies two technical conditions,
and prove that both these conditions hold when the edge is multiple in the 
graph. This leads to recursions for the Chern classes of graph hypersurfaces
for graphs obtained by 
adding parallel edges to a given (regular) edge.

Analogous results for the case of Grothendieck classes of graph hypersurfaces
were obtained in \cite{delecon}. Both Grothendieck classes and Chern classes
were used to define `algebro-geometric' Feynman rules in~\cite{feynman}.
The results in this paper provide further evidence that the polynomial Feynman 
rule defined in terms of the Chern-Schwartz-MacPherson class of a graph 
hypersurface reflects closely the combinatorics of the corresponding graph.

The key to the proof of the main result is a more general formula for the
Chern-Schwartz-MacPherson class of a transversal intersection 
(see \S\ref{generals}), which may be of independent interest.

We also describe a more geometric approach, using the apparatus of
`Verdier specialization'.
\end{abstract}

\maketitle


\section{Introduction}\label{intro}

\subsection{}
{\em Graph hypersurfaces\/} are hypersurfaces of projective space associated with the
parametric formulation of Feynman integrals in scalar quantum field theories. The
study of their geometry was prompted by certain conjectures concerning
the appearance of multiple zeta values in the results of computation of Feynman
amplitudes, and has been the object of intense investigation (see e.g., \cite{MR1435933},
\cite{MR1774975}, \cite{MR1950482}, \cite{MR2238909}, \cite{MR2295619}, 
\cite{MR2604634}, \cite{brownyeats}, \cite{doryn}, \cite{brownschnetz}, and many others).
In this paper we study Chern classes of graph hypersurfaces, from the point of view
of deletion-contraction and multiple-edge formulas.

\subsection{}
`Algebro-geometric Feynman rules' are invariants of graphs which only depend on 
the isomorphism class of the corresponding hypersurfaces, and have controlled
behavior with respect to unions. (This definition of course captures only a very small 
portion of the quantum field theory Feynman rules; it would be very interesting to
have examples of algebro-geometric Feynman rules mirroring more faithfully
their physical counterparts.) Matilde Marcolli and the author note in \cite{feynman}
that the classes of graph hypersurfaces in the Grothendieck ring of varieties
satisfy this basic requirement; Grothendieck classes of graph hypersurfaces have
been studied rather thoroughly, given their relevance to the conjectures mentioned
above. In the same paper we produced a different example of algebro-geometric 
Feynman rules, with values in $\Zbb[t]$, based on the {\em Chern classes\/} of 
graph hypersurfaces. The definition of this polynomial invariant $C_\Gamma(t)$ 
will be recalled below;
its main interest lies in the fact that it carries intersection-theoretic information 
about the singularities of graph hypersurfaces. For example, the (push-forward
to projective space of the) Segre class of the singularity subscheme of a
hypersurface may be recovered from the polynomial Feynman rules. The
Milnor number of the hypersurface (in its natural generalization to arbitrary
hypersurfaces as defined by Parusi\'nski, \cite{MR949831}) is but one piece of 
the information carried by $C_\Gamma(t)$.

The fact that the invariant satisfies the basic requirements of algebro-geometric 
Feynman rules is proved in \cite{feynman}, Theorem~3.6, and is substantially 
less straightforward than the corresponding fact for Grothendieck classes. 

Grothendieck classes of graph hypersurfaces also satisfy a `deletion-contraction'
relation: this fact has been pointed out by several authors, see
e.g.,~\cite{MR1774975}, \cite{MR2238909}, \cite{delecon}.
The purpose of this article is to examine deletion-contraction
relations for Chern classes of graph hypersurfaces, in terms of the polynomial
Feynman rules mentioned above. As in \cite{feynman},
it is natural to expect the situation for Chern classes to be substantially subtler
than for classes in the Grothendieck ring, and in fact our first guess concerning
a double-edge formula for Chern classes, based on somewhat extensive
evidence computed for small graphs, turns out to be {\em incorrect\/} as stated 
in Conjecture~6.1 in~\cite{delecon}. Nevertheless, we will be able to show here 
that these invariants do satisfy the expected general structure underlying 
multiple-edge formulas examined in \cite{delecon}.

\subsection{}
The {\em graph hypersurface\/} associated with a graph $\Gamma$ is the zero-set
of the polynomial
\[
\Psi_\Gamma=\sum_{T} \prod_{e\not\in T} t_e
\]
where $T$ ranges over the maximal spanning forests of $\Gamma$. This is a
homogeneous polynomial in variables $t_e$ corresponding to the edges $e$ of 
$\Gamma$, and its zero set may be viewed in $\Pbb^{n-1}$ or $\Abb^n$,
depending on the context.

Graph hypersurfaces are singular in all but the simplest cases, and in this article
(as in \cite{feynman}) we employ the theory of {\em Chern-Schwartz-MacPherson\/}
(CSM) classes. CSM classes are defined for arbitrary varieties, and agree with
the ordinary (total homology) Chern class of the tangent bundle when evaluated
on nonsingular varieties. The reader may refer to \S2.2-3 of \cite{MR2504753} 
for a quick summary of this theory, which has a long and well-documented history.
CSM classes can be viewed as a generalization
of the topological Euler characteristic: indeed, the degree of the CSM class of a
variety {\em is\/} its Euler characteristic, and to some extent CSM classes maintain
the same additive and multiplicative behavior of the Euler characteristic. In this
respect they are similar in flavor to the Grothendieck class. They also offer a
direct measure of `how singular' a variety is, by comparison with other characteristic
classes of singular varieties, see e.g.~\S2.2 in \cite{MR2504753} and references therein.

CSM classes are in fact defined for {\em constructible functions\/} on a variety
(\cite{MR0361141}), and what we call the CSM class of $X$ is the class $\csm(\one_X)$
of the constant function~$\one_X$. As our objects of study are hypersurfaces of
projective space, we view CSM classes as elements of the Chow group of
projective space, i.e., as polynomials in the hyperplane class. The polynomial
Feynman rules mentioned above are closely related to the CSM class of the
{\em complement\/} of a graph hypersurface $X_\Gamma\subseteq \Pbb^{n-1}$:
if a graph $\Gamma$ with $n$ edges is not a forest, then the polynomial Feynman
rules $C_\Gamma(t)$ are determined by the relation
\[
\csm(\one_{\Pbb^{n-1}\smallsetminus X_\Gamma})=\big(H^n C_\Gamma(1/H)\big)
\cap [\Pbb^{n-1}]\quad,
\]
where $H$ denotes the hyperplane class; see Prop.~3.7 in \cite{feynman}.

\subsection{}
In graph theory, {\em deletion-contraction\/} formulas express invariants for
a graph $\Gamma$ directly in terms of invariants for the `deletion' graph 
$\Gamma\smallsetminus e$ obtained by removing an edge $e$, and the
`contraction' graph $\Gamma/e$ obtained by contracting the same edge.
{\em Tutte-Grothendieck\/} invariants are the most general invariants with
controlled behavior with respect to deletion-contraction; an impressive
list of important graph invariants are of this kind, ranging from chromatic
polynomials to partition functions for Potts models. In fact, these invariants
may be viewed as `Feynman rules' in a sense closely related to the one
adopted in \cite{feynman}, see Prop.~2.2 in~\cite{delecon}.

In \cite{delecon} we show that the invariant arising from the Grothendieck
class satisfies a weak form of deletion-contraction (which involves 
`non-combinatorial' terms); and that enough of this structure is preserved to
trigger combinatorial `multiple-edge' formulas. More precisely, let 
$\Ubb(\Gamma)=[\Abb^n-\hat X_\Gamma]$ denote the Grothendieck class
of the complement of the affine graph hypersurface, and denote by $\Gamma_{2e}$
the graph obtained by doubling the edge $e$ in $\Gamma$. If $e$ is neither
a bridge nor a looping edge, then (\cite{delecon}, Proposition~5.2) 
\begin{equation*}
\tag{*}
\Ubb(\Gamma_{2e})=(\Tbb-1)\,\Ubb(\Gamma)+\Tbb\, \Ubb(\Gamma\smallsetminus e)
+(\Tbb+1)\, \Ubb(\Gamma/e)\quad,
\end{equation*}
where $\Tbb$ is the class of $\Abb^1\smallsetminus \Abb^0$. Note that (*) 
holds without further requirements on~$e$.
(Simpler formulas hold in case $e$ is a bridge or a looping edge.)

\subsection{}
As we will show in this paper, the situation concerning the polynomial 
invariant~$C_\Gamma(t)$ recalled above is somewhat different. As in the case of the 
Grothendieck class, this invariant does not satisfy on the nose a deletion-contraction 
relation (this was already observed in \cite{delecon}, Prop.~3.2). Unlike in the case of 
the invariant~$\Ubb(\Gamma)$, however, even a weaker non-combinatorial form of 
deletion-contraction
only holds under special hypotheses on the edge $e$. The main result of this article
is the determination of conditions on a pair $(\Gamma, e)$ such that a sufficiently
strong deletion-contraction relation (and corresponding consequences, such as
multiple-edge formulas) holds for the edge $e$ of~$\Gamma$. These conditions
are presented in \S\ref{condis}; somewhat surprisingly, they appear to hold for
`many' graphs. One of them can be formulated as follows. Assume that $e$ is
neither a bridge nor a looping edge of $\Gamma$. We may consider the polynomial 
$\Psi_{\Gamma\smallsetminus e}$ for the deletion 
$\Gamma\smallsetminus e$. The condition is then  
that $\Psi_\Gamma$ 
{\em belongs to the Jacobian ideal of $\Psi_{\Gamma\smallsetminus e}$.\/}
The smallest counterexample to this requirement appears to be the graph
\begin{center}
\includegraphics[scale=.5]{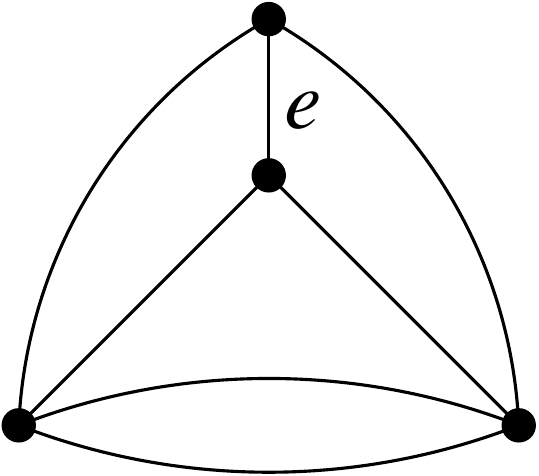}
\end{center}
with respect to the vertical edge. We find it surprising
that this condition is satisfied as often as it is.

The second condition is more technical. See \S\ref{condis} for a discussion of
both conditions.

Once a (weak) deletion-contraction formula is available, one should expect
combinatorial multiple-edge formulas to hold. And indeed, we will prove the
following analogue of (*) for the Chern class Feynman rules:

\begin{theorem}\label{double}
If the conditions on $(\Gamma,e)$ mentioned above are satisfied, then
\[
C_{\Gamma_{2e}}(t)=(2t-1)\, C_\Gamma(t)
-t(t-1)\, C_{\Gamma\smallsetminus e}(t)+C_{\Gamma/e}(t)
\quad.
\]
\end{theorem}

This is the formula that was proposed in \cite{delecon}, Conjecture~6.1, on the 
basis of many examples computed explicitly in \cite{stryker}. However, the
additional conditions on $(\Gamma,e)$ had not been identified at the
time; the formula proposed in {\em loc.~cit.\/} for the class of a triangle with
doubled edges is incorrect, and will be corrected here in Example~\ref{doubletriangex}.

\subsection{}
As mentioned above, we do not have a sharp combinatorial characterization
on $(\Gamma,e)$ ensuring that the technical hypotheses needed for 
Theorem~\ref{double} are satisfied. However, there is one important case
in which we are able to prove that these conditions are indeed satisfied:
the conditions hold if $e$ is itself a multiple edge, i.e., if the endpoints of
$e$ are adjacent in $\Gamma\smallsetminus e$. Thus, 
Theorem~\ref{double} implies that the polynomial Feynman rules satisfy
essentially the same recursive structure for multiple-edge formulas that
is studied in general in \cite{delecon}, \S6. If $e$ is neither a bridge nor a 
looping edge of~$\Gamma$, then
\[
C_{\Gamma^{(m+3)}}(t)=
(3t-1)\, C_{\Gamma^{(m+2)}}(t) - (3t^2-2t)\, C_{\Gamma^{(m+1)}}(t) 
+ (t^3-t^2)\, C_{\Gamma^{(m)}}(t)\quad.
\]
In a sense, this recursion is `nicer' than the corresponding one for 
Grothendieck classes (see the comments following Lemma~\ref{recurs}).

\subsection{}
This paper is organized as follows. In~\S\ref{condis} we state precisely the technical 
conditions mentioned above, in the case of graph hypersurfaces, providing a few
simple examples to illustrate them.
This is also done in the hope that others may identify sharp combinatorial versions
of these conditions. We prove (Lemma~\ref{mainlemma}) that the conditions 
hold for $(\Gamma, e)$ if $e$ has parallel edges in $\Gamma$, and describe one
class of examples in which the conditions do not (both) hold. 
In~\S\ref{generals} we discuss a formula 
for the CSM class of a transversal intersection, needed for the application to
graph hypersurfaces presented here; this section can be read independently
of the rest of the paper. In~\S\ref{particulars} we apply these formulas
to the case of graph hypersurfaces, obtaining the deletion-contraction relation
(Theorem~\ref{delconthm}). In~\S\ref{multedges} we apply this relation to obtain 
multiple-edge formulas as mentioned above (Theorem~\ref{goodform},
Lemma~\ref{recurs}).
In~\S\ref{speci} we describe a different and more `geometric' (but in practice 
less applicable) approach to the main deletion-contraction formula, using 
{\em Verdier's specialization.\/}
\medskip

{\em Acknowledgment.}  
I thank Matilde Marcolli for stimulating my
interest in graph hypersurfaces through our previous joint work, and
for the hospitality at Caltech, where this paper was written. I also
thank Don Zagier for a conversation concerning a technical point
in~\S\ref{multedges}.


\section{Two technical conditions}\label{condis}

We work over an algebraically closed field $k$ of characteristic~zero.

\subsection{}\label{condintro}
As in \S\ref{intro}, $\Gamma$ denotes a finite graph, with $n$ edges; we allow 
looping edges as well as parallel edges. 
We associate with each edge $e$ a 
variable $t_e$, and we consider the {\em graph polynomial\/}
\[
\Psi_\Gamma=\sum_{T} \prod_{e\not\in T} t_e\quad,
\]
where $T$ ranges over the maximal spanning forests of $\Gamma$.
(Note: According to this definition, the polynomial for a graph is the product
of the polynomials for its connected components.)

We denote by $X_\Gamma$ the projective hypersurface defined by $\Psi_\Gamma=0$.
We present in this section two conditions for a pair $(\Gamma,e)$, where $e$ is
an edge of $\Gamma$, encoding some geometric features of $X_\Gamma$.
Finding more transparent, combinatorial versions of these conditions
is an interesting problem.

We will say that an edge $e$ of $\Gamma$ is {\em regular\/} if $e$ is neither a
bridge nor a looping edge, and further $\Gamma\smallsetminus e$ is not a 
forest. We will essentially always assume that $e$ is regular on $\Gamma$;
non-regular edges are easy to treat separately (see e.g., \S\ref{trivca}).

If $e$ is a regular edge of $\Gamma$, then 
\[
\Psi_\Gamma = t_e \Psi_{\Gamma\smallsetminus e} + \Psi_{\Gamma/e}\quad;
\]
this is well-known, and easily checked. As $\Psi_{\Gamma\smallsetminus e}$ is 
not a forest, $\deg \Psi_{\Gamma\smallsetminus e}>0$; in this case a point $p$ of 
$X_\Gamma$ is determined by setting all variables except $t_e$ to $0$. We denote by 
$\Til X_\Gamma$ the blow-up of $X_\Gamma$ at~$p$, and by $E$ the exceptional
divisor of this blow-up. The variety $\Til X_\Gamma$ may be realized as a hypersurface
in the blow-up of $\Pbb^{n-1}$ at~$p$. Denoting by $D$ the exceptional divisor of
this latter blow-up, $E$ is the intersection $D\cap \Til X_\Gamma$. 
Heuristically, the conditions we will present below amount to requiring this intersection
to be sufficiently transversal.

\subsection{}\label{condI}
Assume $e$ is regular. 
Both $X_{\Gamma\smallsetminus e}$
and $X_{\Gamma/e}$ are hypersurfaces of a projective space $\Pbb^{n-2}$
with homogeneous coordinates corresponding to the edges of $\Gamma$
other than~$e$. The first condition on the pair $(\Gamma,e)$
may be expressed as a relation between them:
\begin{equation*}
\tag{Condition I} \Psi_{\Gamma/e}\in (\partial \Psi_{\Gamma\smallsetminus e})\quad.
\end{equation*}
Here, $(\partial \Psi_{\Gamma\smallsetminus e})$ denotes the ideal of partial 
derivatives of $\Psi_{\Gamma\smallsetminus e}$, defining the {\em singularity
subscheme\/} $\partial X_{\Gamma\smallsetminus e}$ of 
$X_{\Gamma\smallsetminus e}$. The condition essentially (that is, up to 
saturating $(\partial \Psi_{\Gamma\smallsetminus e})$) amounts to 
requiring $\partial X_{\Gamma\smallsetminus e}$ to be a subscheme of
$X_{\Gamma/e}$.

It is of course easy to verify whether this condition holds on any given graph,
by employing a symbolic manipulation package such as Macaulay2 (\cite{M2}). 
The following examples illustrate a few cases, showing in particular that the 
condition depends on global features of the graph.

\begin{example}\label{cIexample}
For the graph
\begin{center}
\includegraphics[scale=.4]{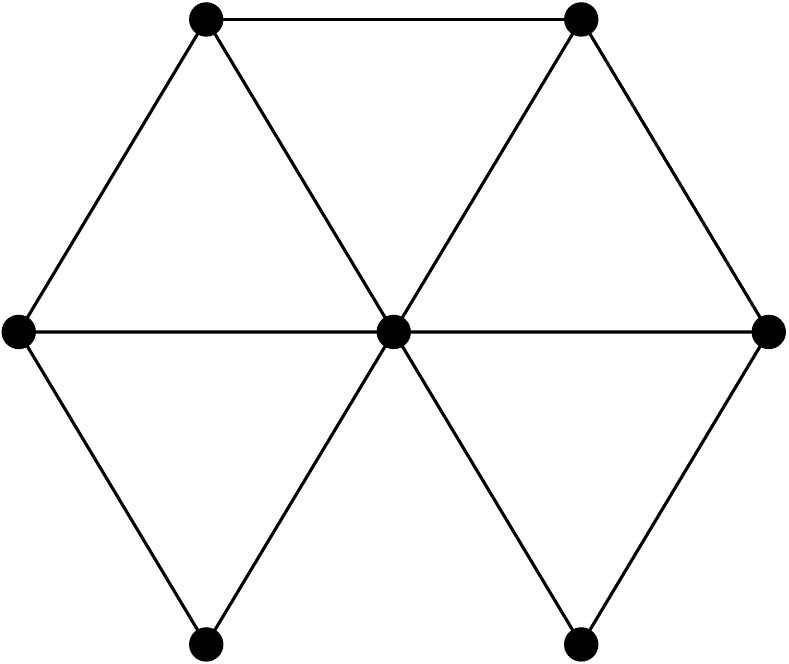}
\end{center}
condition~I is satisfied with respect to all edges.
For the wheel
\begin{center}
\includegraphics[scale=.4]{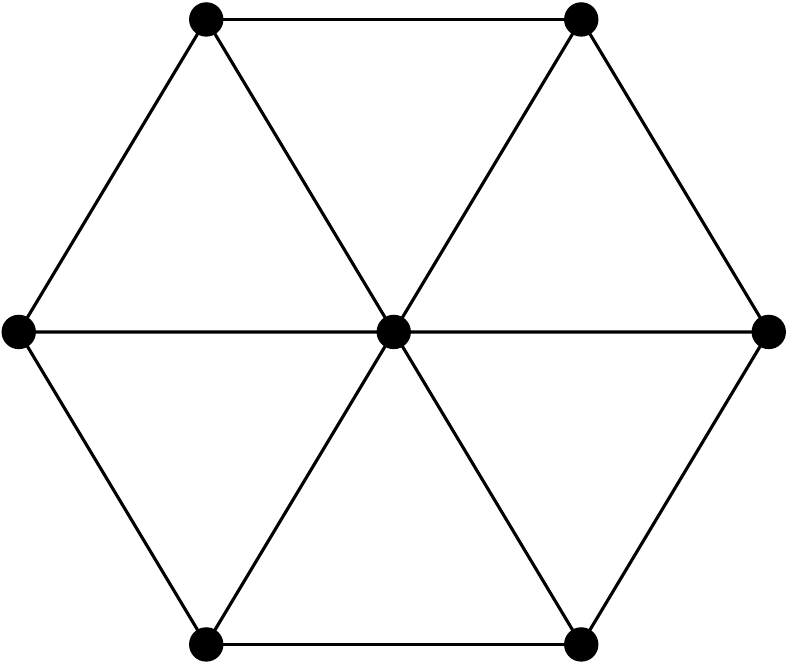}
\end{center}
condition~I is satisfied with respect to the spokes, and it is not satisfied with respect 
to the rim edges.

Condition~I is satisfied with respect to all edges for the graph
\begin{center}
\includegraphics[scale=.5]{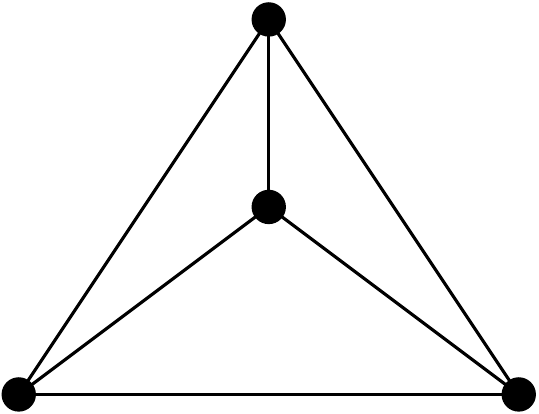}
\end{center}
and with respect to all edges {\em except $e$\/} for the graph
\begin{center}
\includegraphics[scale=.5]{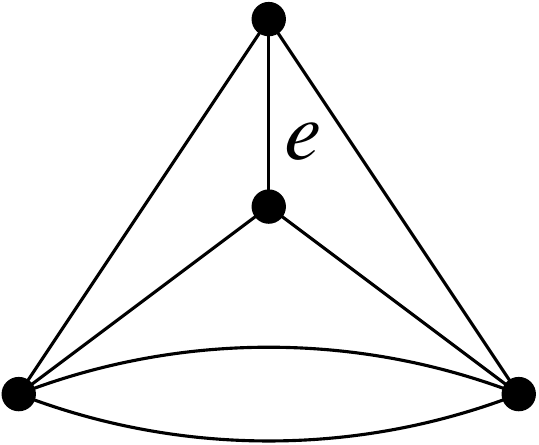}
\end{center}
This is the `smallest' example not satisfying condition~I.
\qede
\end{example}

\subsection{}\label{condII}
The second condition we will consider is more technical than condition~I. 
Let $e$ be a regular edge on $\Gamma$, and consider the blow-up introduced
in \S\ref{condintro}.
For any point $q$ of $E\cap \partial \Til X_\Gamma$, let $I$ be the ideal of 
$\partial \Til X$ at~$q$, and denote by $u$ an equation for $E$ at $q$.
The second condition we must consider on the pair $(\Gamma,e)$~is
\begin{equation*}
\tag{Condition II} \text{For all $q\in E\cap \partial \Til X_\Gamma$, $u$ is a 
non-zero-divisor~modulo $I^j$ for $j\gg 0$}\quad.
\end{equation*}

Again, checking this condition on any given graph is possible with a tool 
such as Macaulay2, although computing power will limit the size of
graphs that can be analyzed in practice. (Note that, by Artin-Rees, only finitely 
many $j$ need be checked.) 

\begin{example}\label{cIIexample}
Condition~II is verified for the graph
\begin{center}
\includegraphics[scale=.5]{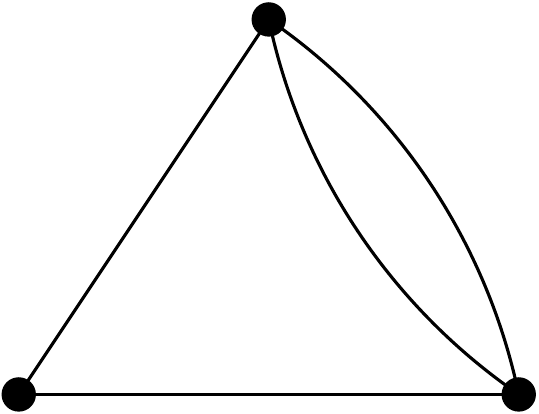}
\end{center}
with respect to all edges, while it does {\em not\/} hold for
\begin{center}
\includegraphics[scale=.5]{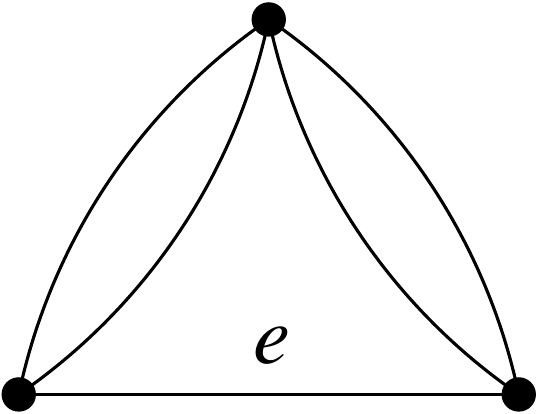}
\end{center}
with respect to $e$. (Both assertions may be verified with Macaulay2; also, see
\S\ref{disjoinable} for a generalization.)
\qede
\end{example}

\subsection{}
We would be interested in purely combinatorial interpretations in terms of $\Gamma$ 
and $e$ of the conditions presented in~\S\ref{condI} and~\S\ref{condII}; it is not even
clear to us that such sharp characterizations exist. However, we can provide one
combinatorial situation in which both conditions are satisfied, and this situation is
at the root of the application to multiple-edge formulas in \S\ref{multedges}.

\begin{lemma}\label{mainlemma}
Let $\Gamma$ be a graph, and let $e$ be a regular edge that has parallel edges 
in~$\Gamma$. Then both conditions~I and~II are verified for $(\Gamma, e)$.
\end{lemma}

\begin{proof}
Let $f$ be an edge parallel to $e$. We first assume that $f$ is not a bridge in 
$\Gamma\smallsetminus e$. Then
\[
\Psi_{\Gamma}=t_e\Psi_{\Gamma\smallsetminus e} +\Psi_{\Gamma/e}
=t_e (t_f \Psi_{\Gamma'}+\Psi_{\Gamma''})+ t_f \Psi_{\Gamma''}
\]
where $\Gamma'=(\Gamma\smallsetminus e)\smallsetminus f$ and $\Gamma''
=(\Gamma/e)/f$.
\begin{center}
\includegraphics[scale=.5]{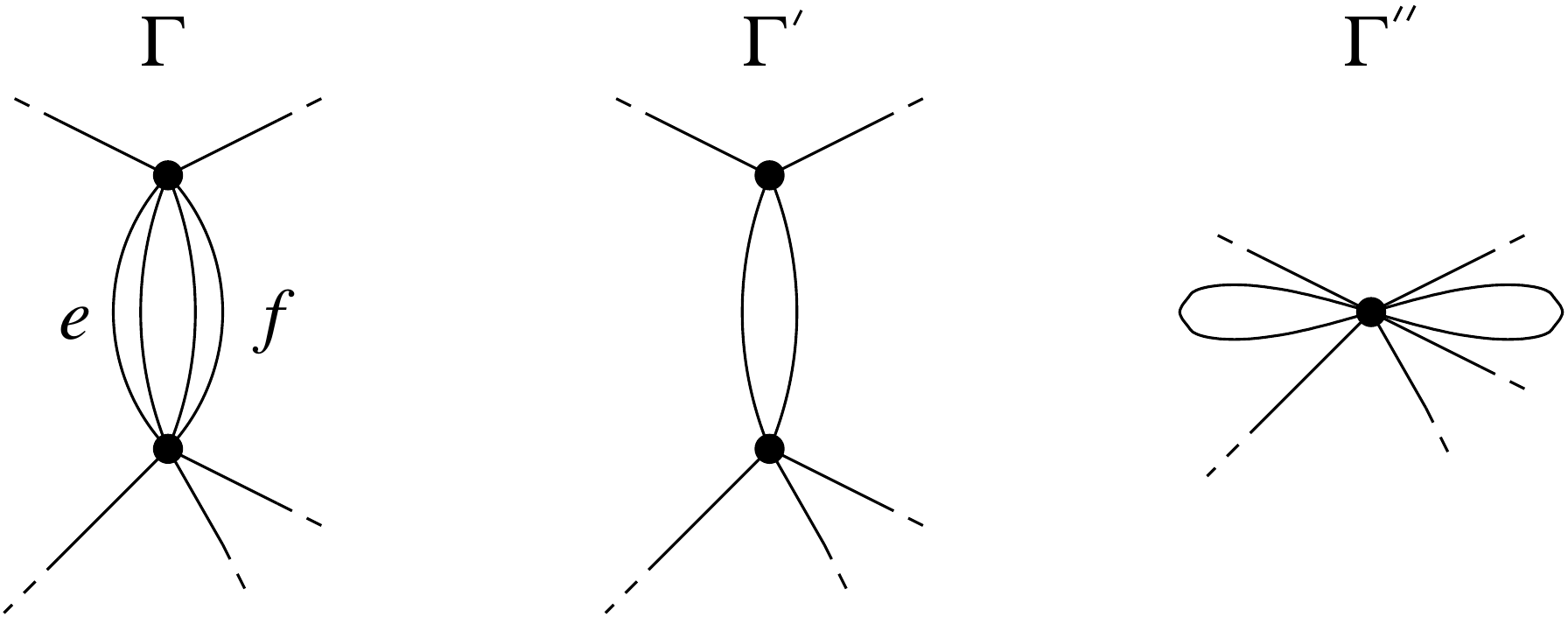}
\end{center}

{\em Condition~I.\/}
Among the partials of $\Psi_{\Gamma\smallsetminus e}$ is $\Psi_{\Gamma'}
=\frac{\partial \Psi_{\Gamma\smallsetminus e}}{\partial t_f}$. Since 
$\Psi_{\Gamma\smallsetminus e}$ is homogeneous (and we are in characteristic zero), 
it is in the ideal of partials
$(\partial \Psi_{\Gamma\smallsetminus e})$. It follows that so is $\Psi_{\Gamma''}$,
and as a consequence $\Psi_{\Gamma/e}=t_f\Psi_{\Gamma''}
\in (\partial \Psi_{\Gamma\smallsetminus e})$
as needed.

{\em Condition~II.\/}
Since the condition only depends on the part of $\Til X_\Gamma$ over $p$, it is 
unaffected by analytic changes of coordinates at $p$. First, we can center an affine
chart at $p$ by setting $t_e=1$, and the equation of $X_\Gamma$ in this chart is
\[
t_f \Psi_{\Gamma'}+ (1+t_f) \Psi_{\Gamma''}=0\quad.
\]
Next, we can set $t_f=\frac \tau{1-\tau}$, i.e., $\tau=\frac{t_f}{1+t_f}$; this does not
affect the geometry of $X_\Gamma$ near~$p$ (where $t_f=0$). In coordinates $\tau,
t_{e_i}$, the equation for $X_\Gamma$ is
\[
\tau \Psi_{\Gamma'} + \Psi_{\Gamma''}=0\quad.
\]
This equation is {\em homogeneous,\/} so $X_\Gamma$ 
is a cone with vertex at $p$ in these coordinates. 
It is then clear that condition~II holds: the equations
of $\Til X_\Gamma$ in the standard charts do not depend on the variable $u$
defining the exceptional divisor, so at each $q\in E\cap \partial \Til X_\Gamma$
the ideal $I$ has a set of generators independent of $u$, and it follows that
$u$ is a non-zero-divisor modulo $I^j$ for all $j$. 

This concludes the proof in the case in which $f$ is not a bridge in 
$\Gamma\smallsetminus e$. If $f$ is a bridge in $\Gamma\smallsetminus e$, 
then $\Psi_{\Gamma'}=\Psi_{\Gamma''}$, and one verifies easily that
\[
\Psi_{\Gamma} = (t_e+t_f) \Psi_{\Gamma'}\quad.
\]
Since $\Psi_{\Gamma\smallsetminus e}=\Psi_{\Gamma'}$ in this case,
it is immediate that $\Psi_{\Gamma}\in (\partial \Psi_{\Gamma\smallsetminus e})$.
Further, $X_\Gamma$ has equation
\[
\Psi_{\Gamma'}=0
\]
in the affine chart $t_e=1$, and near $p$. Again this is a cone with vertex at $p$,
so condition~II holds by the same argument used above.
\end{proof}

\subsection{}\label{transv}
We now discuss more in detail the geometric meaning of the two conditions presented
above. This will also clarify the sense in which the two conditions may be interpreted
as transversality statements.

\begin{claim}\label{condIex}
If Condition~I is satisfied, then $\partial E = E\cap \partial \Til X_\Gamma$.
\end{claim}

\begin{proof}
Recall that $\Til X$ denotes the blow-up of $X$ at the point $p$ obtained by 
setting $t_e$ to $1$ and all other coordinates $t_1,\dots, t_{n-1}$ to $0$
(where $n$ is the number of edges of $\Gamma$, and $e$ is assumed to be
a regular edge).
Working in 
the affine chart $\Abb^{n-1}$ centered at $p$, $X$ has equation
\[
\Psi_{\Gamma\smallsetminus e} + \Psi_{\Gamma/e}=0\quad,
\]
where the summands are homogeneous polynomials of degree $d-1$ and $d$
respectively, with $d=b_1(\Gamma)$. We can cover the blow-up of $\Abb^{n-1}$ 
at $p$ with standard coordinate patches; in one of them we have coordinates
$(u_1,\dots,u_{n-2},u)$ so that the blow-up map is given by
\[
\begin{cases}
t_1 &=u\,u_1 \\
&\dots \\
t_{n-2} &=u\,u_{n-2} \\
t_{n-1} &=u\qquad.
\end{cases}
\]
The equation of $\Til X$ in this chart is then
\[
\Psi_{\Gamma\smallsetminus e}(u_1,\dots,u_{n-2},1) + 
u\, \Psi_{\Gamma/e}(u_1,\dots,u_{n-2},1)=0\quad,
\]
and $u=0$ is the equation of the exceptional divisor in this chart; thus,
$E$ has ideal
\[
(u, \Psi_{\Gamma\smallsetminus e}(u_1,\dots,u_{n-2},1))
\]
in this chart. Note that the exceptional divisor $D$ of the blow-up of $\Abb^{n-1}$
is a projective space $\Pbb^{n-2}$. The above computation (together to the
same in the other patches) shows that $E\cong X_{\Gamma\smallsetminus e}$,
a hypersurface in $D\cong \Pbb^{n-2}$.

This computation also shows that $\partial E=\partial X_{\Gamma\smallsetminus e}$
has ideal $(\partial \Psi_{\Gamma\smallsetminus e})$. In the representative patch
chosen above, this is
\begin{equation*}
\tag{$\dagger$}
\left(u,\,
\Psi_{\Gamma\smallsetminus e},\,
\frac{\partial \Psi_{\Gamma\smallsetminus e}}{\partial u_i}
\right)_{i=1,\dots,n-2}\quad.
\end{equation*}
On the other hand, in the same patch, $\partial \Til X_\Gamma$ has ideal
\[
\left(
\Psi_{\Gamma\smallsetminus e} + 
u\, \Psi_{\Gamma/e},\,
\Psi_{\Gamma/e},\,
\frac{\partial \Psi_{\Gamma\smallsetminus e}}{\partial u_i}
+ u\, \frac{\partial \Psi_{\Gamma/ e}}{\partial u_i}
\right)_{i=1,\dots,n-2}
\]
and hence $E\cap \partial \Til X_\Gamma$ has ideal
\begin{equation*}
\tag{$\ddagger$}
\left(
u,\,
\Psi_{\Gamma\smallsetminus e},\,
\Psi_{\Gamma/e},\,
\frac{\partial \Psi_{\Gamma\smallsetminus e}}{\partial u_i}
\right)_{i=1,\dots,n-2}\quad.
\end{equation*}
Comparing ($\dagger$) and ($\ddagger$) (and the analogous ideals in all 
patches covering the blow-up), 
we see that the ideals agree if $\Psi_{\Gamma/e}\in 
(\partial \Psi_{\Gamma\smallsetminus e})$, that is, if Condition~I holds.
This verifies Claim~\ref{condIex}.
\end{proof}

\begin{remark}
The picture we have in mind is that of the nonsingular exceptional divisor 
$D\cong\Pbb^{n-2}$ intersecting $\Til X$ along $E$. According to 
Claim~\ref{condIex}, Condition~I implies that the intersection 
$E=D\cap \Til X_\Gamma$ is only singular along the intersection of $D$ with 
the singularity subscheme of $\Til X_\Gamma$, as would be expected if
$D$ met $\Til X_\Gamma$ transversally.
\qede\end{remark}

In order to interpret Condition~II, we have to introduce the blow-up 
$\mu: \widehat X_\Gamma \to \Til X_\Gamma$
of $\Til X_\Gamma$ along its singularity subscheme $\partial \Til X_\Gamma$.
In this blow-up we may consider two subschemes: the proper transform
$\widehat E$ (isomorphic to the blow-up of $E$ along $E\cap \partial \Til X$)
and the inverse image $\mu^{-1}(E)$.

\begin{claim}\label{condIIex}
If Condition~II holds, then $\widehat E=\mu^{-1}(E)$.
\end{claim}

\begin{proof}
Indeed, assume Condition~II holds. Then letting $\cI$, $\cJ$ denote respectively
the ideal sheaves of $\partial \Til X_\Gamma$ and $E$ in $\Til X_\Gamma$, Condition~II
implies $\cJ\cap \cI^j =\cJ\cdot \cI^j$ for $j\gg 0$, and hence the natural morphism
\[
\frac{\cI^j}{\cJ\cdot \cI^j} \to
\frac{\cI^j+\cJ}{\cJ}
\]
is an isomorphism for $j\gg 0$. It follows that the natural inclusion
\[
\text{Proj}_{\cO_{\Til X}}\left(\bigoplus_j \frac{\cI^j+\cJ}{\cJ}\right)
=B\ell_{E\cap \partial\Til X} E\,\subseteq\,
\mu^{-1}(E)=
\text{Proj}_{\cO_{\Til X}}\left(\bigoplus_j \frac{\cI^j}{\cJ\cdot \cI^j}\right)
\]
is an equality, verifying Claim~\ref{condIIex}.
\end{proof}

\begin{remark}
By the same token, the proper transform of the exceptional divisor $D$ equals
its inverse image in the blow-up along $\partial \Til X_\Gamma$. This
is the behavior expected if the (nonsingular) hypersurface $D$ meets the
center of a blow-up transversally, so Condition~II, like Condition~I, 
appears to express a measure of transversality of the intersection of
$D$ with $\Til X$.
\qede\end{remark}

\begin{remark}\label{InotII}
The two conditions differ: for example, Condition~II
does not hold for the second graph displayed in Example~\ref{cIIexample}
with respect to edge $e$, while Condition~I does hold in this case.

However, we do not know of examples of graphs for which Condition~II
holds and Condition~I does not. It is conceivable that such examples
exist (cf.~\S\ref{sharps}).
\qede\end{remark}

\subsection{}\label{disjoinable}
Finally, we discuss one case in which conditions~I and~II do not both hold.
If a graph is obtained by taking the union of two graphs $\Gamma'$, 
$\Gamma''$, joined at a vertex, then its graph polynomial equals 
the product $\Psi_{\Gamma'} \cdot \Psi_{\Gamma''}$. If neither $\Gamma'$
nor $\Gamma''$ is a forest, we say that the graph is `disjoinable'; note that
the graph hypersurface of a disjoinable graph is singular in codimension~$1$
(We do not know whether the converse holds.)

\begin{claim}
Let $\Gamma$ be a graph such that $X_\Gamma$ is nonsingular in codimension~$1$.
Let $e$ be an edge such that $\Gamma\smallsetminus e$ is disjoinable.
Then at least one of conditions~I and ~II fails for $(\Gamma,e)$.
\end{claim}

For example, the second graph drawn in Example~\ref{cIIexample} is of this type,
with respect to the bottom edge. (The first is not, since while removing the bottom 
edge does produce the join of two graphs, one of these is a tree.)

\begin{proof}
If condition~I does not hold, we are done; so we may assume that condition~I
holds, and we will show that condition~II does not hold in this case.

We use notation as in the discussion following the statement of Claim~\ref{condIex}.
After setting $t_e$ to $1$, we have
\[
\Psi_\Gamma=\Psi_{\Gamma'}\cdot \Psi_{\Gamma''} + \Psi_{\Gamma/e}
\]
by assumption, where $\Psi_{\Gamma'}$ and $\Psi_{\Gamma''}$ are graph 
polynomials of degree~$\ge 1$. In the chart of $\Til \Abb^{n-1}$ with coordinates 
$(u_1,\dots,u_{n-2},u)$, the equation of $\Til X$ is
\[
\Psi_{\Gamma'}(u_1,\dots,u_{n-2},1)\cdot \Psi_{\Gamma''}(u_1,\dots,u_{n-2},1) 
+ u \,\Psi_{\Gamma/e}(u_1,\dots,u_{n-2},1) =0\quad,
\]
and $E$ has ideal
\[
(u,\Psi_{\Gamma'}(u_1,\dots,u_{n-2},1)\cdot \Psi_{\Gamma''}(u_1,\dots,u_{n-2},1))\quad.
\]
In particular, the singularity subscheme $\partial E$ of $E$ contains the 
locus $Z$ with
ideal $(u,\Psi_{\Gamma'}(u_1,\dots,u_{n-2},1), \Psi_{\Gamma''}(u_1,\dots,u_{n-2},1))$.
As we are assuming that condition~I holds, we have that $\partial E=
E\cap \partial \Til X$ (cf.~Claim~\ref{condIex}); in particular, 
$Z\subseteq \partial \Til X$. On the other
hand, note that $Z$ has codimension~$1$ in $E$, hence codimension~$2$ 
in $\Til X$; since $X$ is nonsingular in codimension~$1$ and 
$Z\subseteq E$, $Z$ must consist of a collection of components of 
$\partial \Til X$. But then $E$ contains components of $\partial \Til X$, and it 
follows that the condition in Claim~\ref{condIIex} is not verified. 
Hence $(\Gamma,e)$ does not satisfy condition~II.
\end{proof}


\section{CSM classes of transversal intersections}\label{generals}

\subsection{}
In this section we discuss a formula expressing the Chern-Schwartz-MacPherson 
class of the intersection of a variety $X$ with a hypersurface $D$, in terms of of
$\csm(X)$ and of the class of $D$. This section can be read independently of
the rest of the paper.

Our template is the transversal intersection of nonsingular varieties. Let $V$ be a
nonsingular variety, and let $D$, $X$ be nonsingular subvarieties of $V$. Assume
that $D$ is a hypersurface, and that $D$ and $X$ meet transversally.
In this case $D\cap X$ is a nonsingular hypersurface of $X$, and $\cO_X(D\cap X)
=\cO_V(D)|_X$, hence (harmlessly abusing notation)
\[
c(T(D\cap X))\cap [D\cap X] = \frac{c(TX)}{c(N_DV)}\cap [D\cap X]
=\frac{D}{1+D} \cap c(TX)\cap [X]\quad,
\]
i.e.,
\[
\csm(D\cap X) = \frac D{1+D} \cap \csm(X)\quad.
\]
(This equality holds in $A_*X$.)
We are interested in generalizing this formula to the case in which $X$ is possibly
singular.

\subsection{}
The key question is, of course, what `transversal' should mean in the singular
case. The conditions presented in \S\ref{condis}
are precisely concocted to make this requirement precise.

We assume $V$ is a nonsingular variety, $D$ and $X$ are reduced hypersurfaces 
of $V$, and $D$ is nonsingular. We denote by $\partial X$ the singularity subscheme 
of $X$, defined by the ideal of partial derivatives of local equations. We denote
by $\rho: \Til V \to V$ the blow-up along $\partial X$. Further, $\Til D$ denotes
the proper transform of $D$ in $\Til V$.

\begin{theorem}\label{adj}
Assume that
\begin{enumerate}
\item $\partial (D\cap X)=D\cap \partial X$;
\item $\Til D=\rho^{-1} D$. 
\end{enumerate}
Then 
\[
\csm(D\cap X) = \frac D{1+D} \cap \csm(X)
\]
in $A_*X$.
\end{theorem}

\begin{proof}
Our main tool is Theorem~I.4 in \cite{MR2001i:14009}, which relates the
Chern-Schwartz-MacPherson class of a hypersurface with the Segre class of its
singularity subscheme. Viewing $D\cap X$ as a hypersurface in $D$, this result
yields
\[
\csm(D\cap X)=c(TD)\cap \left((s(D\cap X, D)+ c(\cO(X))^{-1}\cap
(s(\partial(D\cap X),D)^\vee \otimes_D \cO(X))\right)
\]
(with notation as in \cite{MR2001i:14009}, \S1.4); by the same token,
\[
\csm(X)=c(TV)\cap \left(s(X, V)+ c(\cO(X))^{-1}\cap
(s(\partial X,V)^\vee \otimes_V \cO(X))\right)\quad.
\]
It follows that
\[
\frac D{1+D}\cap \csm(X)
=c(TD)\cap (s(D\cap X,D)+D\cdot c(\cO(X))^{-1}\cap
(s(\partial X,V)^\vee \otimes_V \cO(X))\quad.
\]
Therefore, the formula stated in Theorem~\ref{adj} would follow from the equality
\[
D\cdot c(\cO(X))^{-1}\cap (s(\partial X,V)^\vee \otimes_V \cO(X))
=c(\cO(X))^{-1}\cap (s(\partial(D\cap X),D)^\vee \otimes_D \cO(X))\quad,
\]
and hence from
\[
(D\cdot s(\partial X,V))^\vee \otimes_D \cO(X))
=s(\partial(D\cap X),D)^\vee \otimes_D \cO(X)\quad,
\]
and finally from
\[
D\cdot s(\partial X, V)=s(\partial(D\cap X),D)
\]
(in $A_*X)$. This reduces the proof of Theorem~\ref{adj} to the following claim:

\begin{claim}
Under hypotheses (1) and (2) from the statement of Theorem~\ref{adj},
\[
s(\partial (D\cap X),D) = D\cdot s(\partial X, V)
\]
in $A_*(D\cap X)$.
\end{claim}

To prove this, consider the blow-up of
$V$ along $\partial X$, with exceptional divisor $F$; the proper transform of
$D$ may be viewed as the blow-up of the latter along $D\cap \partial X$:
\[
\xymatrix@!@C=-15pt@R=-15pt{
& F \ar@{-}[d]^(.6)\sigma \ar@{^(->}[rr] & & \Til V \ar[dd]^\rho \\
\Til D\cap F\ar[dd]_{\sigma'} \ar@{^(->}[rr] \ar@{^(->}[ur] & \ar[d] & \Til D \ar[dd] \ar@{^(->}[ur] \\
& \partial X \ar@{^(-}[r] & \ar[r] & V \\
D\cap \partial X \ar@{^(->}[rr] \ar@{^(->}[ur] & & D \ar@{^(->}[ur]
}
\]
By the birational invariance of Segre classes,
\begin{align*}
s(\partial X,V) &=\sigma_* s(F,\Til V)=\sigma_*\left(\frac{F}{1+F}\cap [\Til V]\right)\\
s(D\cap \partial X,D) &=\sigma'_* s(\Til D\cap F,\Til D)
=\sigma'_*\left(\frac{F}{1+F}\cap [\Til D]\right)\quad.
\end{align*}
By hypothesis (2), $\rho^{-1} D=\Til D$, hence $c_1(\sigma^*\cO(D)|_{\partial X})$ is 
represented by $\Til D\cap F$. Therefore, by the projection formula (\cite{85k:14004}, 
Proposition~2.3 (c))
\[
\sigma'_*\left(\frac F{1+F}\cap [\Til D]\right)
=\sigma'_*\left(\frac{[\Til D\cap F]}{1+F}\right)
=\sigma'_*\left(\sigma^* D\cdot \frac{F}{1+F}\cap [\Til V]\right)
=D\cdot s(\partial X,V)\quad.
\]
Thus,
\[
s(D\cap \partial X,D)=D\cdot s(\partial X,V)\quad.
\]
Finally, $D\cap \partial X=\partial (D\cap X)$ by hypothesis (1), and the claim
follows. This concludes the proof of Theorem~\ref{adj}.
\end{proof}

\subsection{}\label{sharps}
Theorem~\ref{adj} is sharp, in the sense that both hypotheses are necessary
for the stated formula to hold.
\smallskip

To see that the first hypothesis is needed, let $X$
be a nonsingular quadric in $\Pbb^3$, and let $D$ be a hyperplane tangent
to $X$.
\begin{center}
\includegraphics[scale=.6]{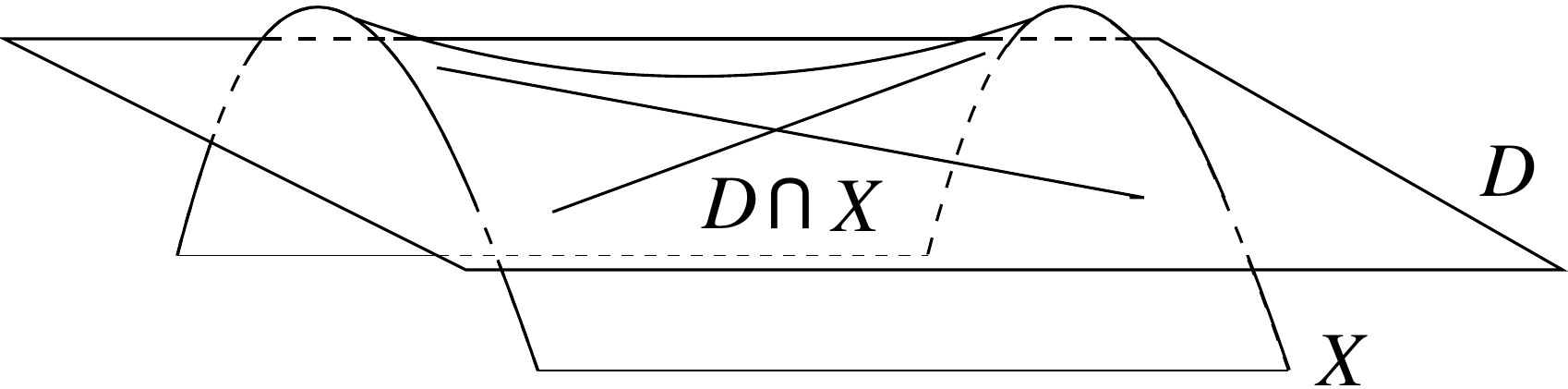}
\end{center}
Then $D\cap X$ is singular, while $\partial X=\emptyset$; in particular, (1) is not 
satisfied. Working in the ambient $[\Pbb^3]$ for convenience, we have
\[
\csm(D\cap X) = 2[\Pbb^1]+3[\Pbb^0]
\]
and 
\[
\frac D{1+D}\cap \csm(X) = \frac H{1+H} \cap (2[\Pbb^2] + 4[\Pbb^1] +4[\Pbb^0])
=2[\Pbb^1]+2[\Pbb^0]
\]
(where $H$ is the hyperplane class). Therefore, the stated formula does not hold.
Note that the second hypothesis holds (trivially) in this case, since 
$\partial X=\emptyset$.
\smallskip

To see that the second hypothesis is also necessary, let $X\subseteq \Pbb^3$
be a quadric cone, and let $D$ be a general hyperplane through the vertex.
\begin{center}
\includegraphics[scale=.6]{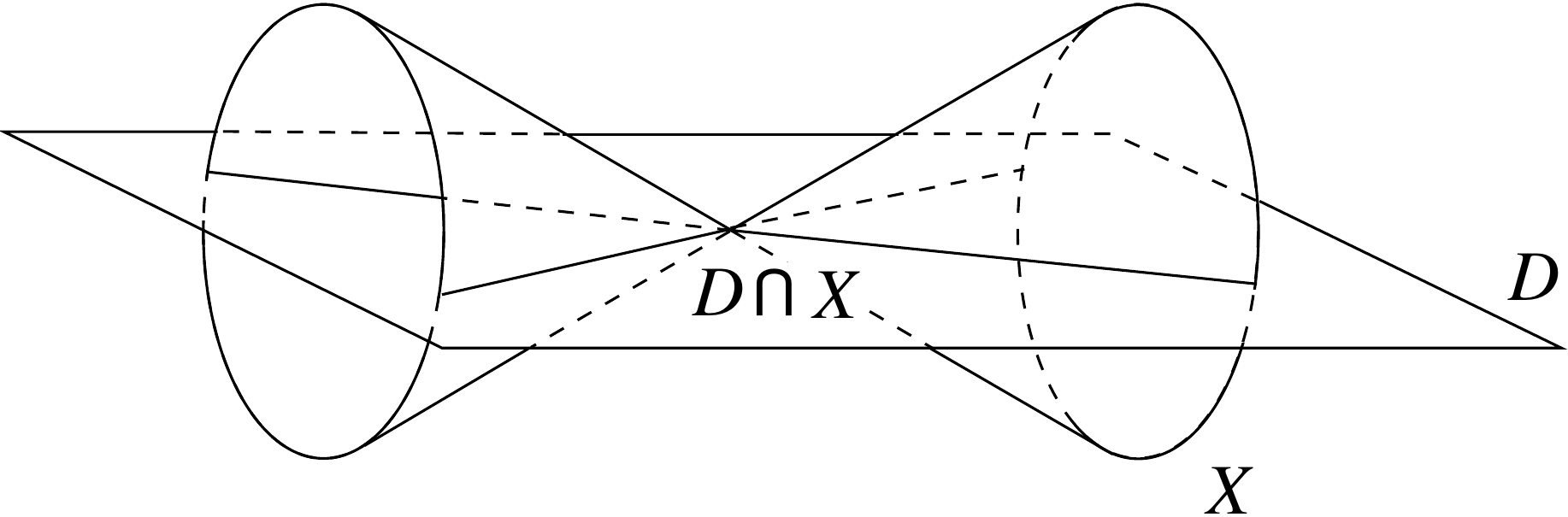}
\end{center}
Then (2) fails as $D$ contains $\partial X$. Again working in $\Pbb^3$, we have
\[
\csm(D\cap X) = 2[\Pbb^1]+3[\Pbb^0]
\]
and
\[
\frac D{1+D}\cap \csm(X) = \frac H{1+H} \cap (2[\Pbb^2] + 4[\Pbb^1] +3[\Pbb^0])
=2[\Pbb^1]+2[\Pbb^0]
\]
as before. (The coefficient of $[\Pbb^0]$ in $\csm(X)$ is in fact irrelevant to
this computation.) Note that $\partial(D\cap X)=\partial X=$ the vertex of the cone, 
and in particular $\partial(D\cap X)=D\cap \partial X$, so that the first hypothesis 
does hold in this case. 


\section{Deletion-contraction for Chern classes of graph hypersurfaces}\label{particulars}

\subsection{}
Now we return to the case of graph hypersurfaces. Recall that $\Gamma$ denotes
a graph with $n$ edges; $e$ denotes a regular edge of $\Gamma$; $X_\Gamma
\subseteq \Pbb^{n-1}$ is the corresponding hypersurface. As $\Gamma\smallsetminus e$
is not a forest, then the point $p$ obtained by setting all coordinates except $t_e$
to zero is a point of $X_\Gamma$.

We consider the blow-up $\Til\Pbb^{n-1} \to \Pbb^{n-1}$ at $p$, and the blow-up
$\Til X_\Gamma$ of $X_\Gamma$ at $p$ realized as the proper transform of 
$X_\Gamma$ in $\Til \Pbb^{n-1}$. We denote $D$ the exceptional divisor in 
$\Til \Pbb^{n-1}$, so that $E=D\cap \Til X_\Gamma$ is the exceptional divisor in 
$\Til X_\Gamma$. Applying Theorem~\ref{adj} to this situation yields:

\begin{corol}\label{maincor}
Assume conditions~I and~II from \S\ref{condis} hold for $(\Gamma,e)$. Then
\[
\csm(E)=\frac D{1+D}\cap \csm(\Til X_\Gamma)\quad.
\]
\end{corol}

Indeed, as observed in \S\ref{transv}, conditions~I and~II imply the hypotheses of 
Theorem~\ref{adj}.

\subsection{}\label{gende}
We are ready to prove a deletion-contraction formula for Chern classes
of graph hypersurfaces, subject to the conditions presented in~\S\ref{condis}.

We can view the blow-up $\Til\Pbb^{n-1}$ as the graph of the projection
$\Pbb^{n-1} \dashrightarrow \Pbb^{n-2}$ centered at $p$:
\[
\xymatrix@C=10pt{
& \Til\Pbb^{n-1} \ar[dl] \ar[dr] \\
\Pbb^{n-1} \ar@{-->}[rr] & &  \Pbb^{n-2}
}
\]
and likewise for $\Til X_\Gamma$:
\[
\xymatrix@C=10pt{
& \Til X_\Gamma \ar[dl]_\nu \ar[dr]^\pi \\
X_\Gamma \ar@{-->}[rr] & &  \Pbb^{n-2}
}
\]
This situation was briefly described in \cite{delecon}, end of \S4: $\nu$ is the blow-up 
of $X_\Gamma$ at~$p$, while $\pi$ realizes $\Til X_\Gamma$ as the blow-up
of $\Pbb^{n-2}$ along $X_{\Gamma\smallsetminus e}\cap X_{\Gamma/e}$;
the fibers of $\pi$ are points away from $X_{\Gamma\smallsetminus e}
\cap X_{\Gamma/e}$, and $\Pbb^1$ over over points of 
$X_{\Gamma\smallsetminus e}\cap X_{\Gamma/e}$.
The exceptional divisor $E$ of $\nu$ is a copy of $X_{\Gamma\smallsetminus e}$
in $D\cong \Pbb^{n-2}$ (as was verified in~\S\ref{transv}). The restriction of
$\Til \Pbb^{n-1}\to \Pbb^{n-2}$ to $D$ gives an isomorphism $D\to \Pbb^{n-2}$.

We are aiming for formulas involving the polynomial `Feynman rules' 
$C_\Gamma(t)$ carrying the information of the CSM class of $X_\Gamma$.
Denote by $H$ the hyperplane class in $\Pbb^{n-1}$.
Our main objective is essentially a formula for the polynomial
\[
\sum_{i\ge 0} t^i \int H^i\cap \csm(X_\Gamma)
\]
encoding the degrees of the terms in $\csm(X_\Gamma)$; $C_\Gamma(t)$ may be
computed easily from this polynomial.

\subsection{}
We let $H$, resp.~$h$, denote the class in $\Pbb^{n-1}$, resp.~$\Pbb^{n-2}$.
The Chow group of $\Til X_\Gamma$ is generated by $\pi^*(h)$ and the class
$E$ of the exceptional divisor.

\begin{lemma}\label{tech1}
With notation as above,
\[
\sum_{i\ge 0} t^i \int H^i\cap \csm(X_\Gamma)
=1-\chi(X_{\Gamma\smallsetminus e})+\int \frac {1+tE}{1-t\,\pi^*h}
\cap \csm(\Til X_\Gamma)
\quad.
\]
\end{lemma}

\begin{proof}
By the functoriality property of CSM classes,
\begin{align*}
\nu_*(\csm (\Til X_\Gamma)) &= \nu_* (\csm(\one_{\Til X_\Gamma}))
=\csm(\nu_* \one_{\Til X_{\Gamma}})
=\csm(\one_{X_\Gamma}+(\chi(E)-1) \one_p) \\
&=\csm(X_\Gamma)+(\chi(E)-1)[p]\quad.
\end{align*}
Using that $E\cong X_{\Gamma\smallsetminus e}$, and applying the
projection formula, this gives
\begin{align*}
\sum_{i\ge 0} t^i \int H^i\cap \csm(X_\Gamma)
&=1-\chi(X_{\Gamma\smallsetminus e})
+\sum_{i\ge 0} \nu_*\left(t^i \int (\nu^* H)^i \cap \csm(\Til X_\Gamma)\right) \\
&=1-\chi(X_{\Gamma\smallsetminus e})
+\int \frac{1} {1-t\, \nu^* H} \cap \csm(\Til X_\Gamma)
\quad,
\end{align*}
with the last equality due to the fact that push-forwards preserve degrees,
and condensing the summation into a rational function for notational convenience.

Now I claim that $\nu^*H=E+\pi^*h$: indeed, this may be verified by realizing
$H$ as the class of a general hyperplane containing $p$. Also, note that 
$E\cdot \nu^*H=0$: realize $H$ as the class of a hyperplane {\em not\/} 
containing $p$ to verify this. Therefore,
\[
\frac 1{1-t\, \nu^* H} -\frac {1+t E}{1-t\, \pi^*h}
=\frac {t^2 E\cdot \nu^*H}{(1-t\,\nu^* H)(1-t\,\pi^* h)} = 0\quad.
\]
The statement follows.
\end{proof}

\begin{lemma}\label{tech2}
\[
\int \frac 1{1-t\,\pi^*h}
\cap \csm(\Til X_\Gamma)=\frac{(1+t)^{n-1}-1}t
+\sum_{i\ge 0} t^i\int h^i\cap \csm(X_{\Gamma\smallsetminus e}\cap X_{\Gamma/e})\quad.
\]
\end{lemma}

\begin{proof}
By the projection formula, and since push-forwards preserve degree,
\[
\int \frac 1{1-t\,\pi^*h} \cap \csm(\Til X_\Gamma)=
\sum_{i\ge 0} t^i \int h^i \cap \pi_*(\csm (\Til X_\Gamma))\quad.
\]
Applying the functoriality of CSM classes:
\begin{align*}
\pi_*(\csm (\Til X_\Gamma)) &= \pi_* (\csm(\one_{\Til X_\Gamma}))
=\csm(\pi_* \one_{\Til X_\Gamma})
=\csm(\one_{\Pbb^{n-2}}+\one_{X_{\Gamma\smallsetminus e}
\cap X_{\Gamma/e}})\\
&=c(T\Pbb^{n-2})\cap [\Pbb^{n-2}] + \csm(X_{\Gamma\smallsetminus e}
\cap X_{\Gamma/e})\quad,
\end{align*}
where we have used the description of the fibers of $\pi$ recalled in~\S\ref{gende}.
As $c(T\Pbb^{n-2})=(1+h)^{n-1}-h^{n-1}$, the statement follows.
\end{proof}

\subsection{}
Combining Lemma~\ref{tech1} and~\ref{tech2}, we obtain that {\em if $e$ is a
regular edge on $\Gamma$, then\/}
\begin{multline*}
\sum_{i\ge 0} t^i \int H^i\cap \csm(X_\Gamma)
=1-\chi(X_{\Gamma\smallsetminus e})+\frac{(1+t)^{n-1}-1}t \\
+\sum_{i\ge 0} t^i\int h^i\cap \csm(X_{\Gamma\smallsetminus e}\cap X_{\Gamma/e})
+\int \frac {t\,E}{1-t\,\pi^*h}
\cap \csm(\Til X_\Gamma)\quad,
\end{multline*}
The more technical conditions~I and~II presented in \S\ref{condis} play no role in
this statement. They become relevant in evaluating the last term,
\[
\int \frac {t\,E}{1-t\,\pi^*h} \cap \csm(\Til X_\Gamma)
=\int \frac {t\,D}{1-t\,h} \cap \csm(\Til X_\Gamma)\quad.
\]
here we have replaced $E$ by $D$ (since $E=D\cap \Til X$,
and in particular $D$ restricts to the class of $E$ on $\Til X_\Gamma$),
and $\pi^* h$ with the corresponding hyperplane class $h$ on $D\cong \Pbb^{n-2}$.

The class $D\cap \csm(\Til X_\Gamma)$ is supported on 
$E\cong X_{\Gamma\smallsetminus e}$. Without further information on $e$
and $\Gamma$, it does not seem possible to express this class in more
intelligible terms.

\begin{lemma}\label{tech3}
Assume $(\Gamma,e)$ satisfies conditions~I and~II. Then
\[
\int \frac {t\,D}{1-t\,h} \cap \csm(\Til X_\Gamma)=
\chi(X_{\Gamma\smallsetminus e})+(t-1)\sum_{i\ge 0} t^i\int h^i
\cap \csm(X_{\Gamma\smallsetminus e})
\]
\end{lemma}

\begin{proof}
By Corollary~\ref{maincor},
\[
D\cap \csm(\Til X_\Gamma)=(1+D)\cap \csm(E)=(1-h)\cap 
\csm(X_{\Gamma\smallsetminus e})\quad:
\]
here we have identified $E\subseteq D$ with $X_{\Gamma\smallsetminus e}
\subseteq \Pbb^{n-2}$, and used the fact that the class of the exceptional
divisor $D$ restricts to $\cO(-1)$. The statement is obtained by applying 
mindless manipulations (and noting $\int \csm(X_{\Gamma\smallsetminus e})
=\chi(X_{\Gamma\smallsetminus e}))$:
\begin{align*}
\int \frac {t\,D}{1-t\,h} \cap \csm(\Til X_\Gamma)
&=\int \frac{t(1-h)}{1-t\,h}\cap \csm(X_{\Gamma\smallsetminus e})
=\int \left(1+\frac{(t-1)}{1-t\,h}\right)\cap \csm(X_{\Gamma\smallsetminus e})\\
&=\chi(X_{\Gamma\smallsetminus e})+(t-1)\int\frac 1{1-t\,h}
\cap \csm(X_{\Gamma\smallsetminus e})
\end{align*}
with the stated result.
\end{proof}

\subsection{}
Collecting what we have proved at this point:
\begin{prop}\label{delconprop}
Let $\Gamma$ be a graph with $n$ edges, and let $e$ be a regular edge of
$\Gamma$. 
Assume $(\Gamma,e)$ satisfies the conditions given in~\S\ref{condis}.
Then
\begin{multline*}
\sum_{i\ge 0} t^i \int H^i\cap \csm(X_\Gamma)
=\frac{(1+t)^{n-1}+(t-1)}t\\
+\sum_{i\ge 0} t^i\int h^i\cap \csm(X_{\Gamma\smallsetminus e}\cap X_{\Gamma/e})
+(t-1)\sum_{i\ge 0} t^i\int h^i \cap \csm(X_{\Gamma\smallsetminus e})
\end{multline*}
\end{prop}
This statement improves considerably once it is expressed in terms of the `polynomial 
Feynman rules' introduced in \cite{feynman}, \S3; we take this as a further indication
that the polynomial captures interesting information about $\Gamma$.
The polynomial essentially evaluates the CSM class of the {\em complement\/}
of $X_\Gamma$ in projective space. We will denote the polynomial 
corresponding to $\Gamma$ by $C_{X_\Gamma}(t)$, since it depends directly on 
the graph hypersurface (this makes it an {\em algebro-geometric\/} Feynman rule),
and since it can be defined for any subset of projective space.

\begin{lemma}\label{poly}
If $\Gamma$ is not a forest and has $n$ edges, then
\[
C_{X_\Gamma}(t)=(1+t)^n-1-\sum_{i\ge 0} t^{i+1} \int H^i\cap \csm(X_\Gamma)\quad.
\]
\end{lemma}

This is obtained from Proposition~3.7 of \cite{feynman}, by applying simple
manipulations. If $\Gamma$ is a forest, then $X_\Gamma$ is empty, and the 
corresponding polynomial
is a power of $(t+1)$, cf.~Proposition~3.1 in~\cite{feynman}.

Consistently with the expression in Lemma~\ref{poly}, we set
\[
C_{Z}(t)=(1+t)^{n-1}-1-\sum_{i\ge 0} t^{i+1} \int h^i\cap \csm(Z)
\]
for every nonempty subscheme $Z\subseteq \Pbb^{n-2}$, where $h$ denotes the 
hyperplane class. Note that $C_Z(t)$ depends on the dimension of the space 
containing $Z$; this is always clear from the context. If $Z\subset \Pbb^{n-1}$ is 
empty, we set $C_Z(t)=(1+t)^{n-1}$.

\begin{theorem}[Deletion-contraction]\label{delconthm}
Let $e$ be a regular edge of $\Gamma$. Assume $(\Gamma,e)$ satisfies 
both conditions~I and~II given in~\S\ref{condis}. Then
\[
C_{X_\Gamma}(t)=C_{X_{\Gamma\smallsetminus e}\cap X_{\Gamma/e}}(t)
+(t-1)\, C_{X_{\Gamma\smallsetminus e}}(t)\quad.
\]
\end{theorem}

This is the form taken by the formula in Proposition~\ref{delconprop}, once it is 
written using the notation recalled above. On the right-hand side, both 
$X_{\Gamma\smallsetminus e}$ and $X_{\Gamma\smallsetminus e}
\cap X_{\Gamma/e}$ are viewed as subschemes of $\Pbb^{n-2}$.

The deletion-contraction formula of Theorem~\ref{delconthm} holds also if
$\Gamma\smallsetminus e$ is a forest, provided that 
$C_{X_{\Gamma\smallsetminus e}\cap X_{\Gamma/e}}(t)$ and 
$C_{X_{\Gamma\smallsetminus e}}(t)$ are both taken to equal $(t+1)^{n-1}$.

\begin{remark}\label{myst}
Differentiating the formula in Theorem~\ref{delconprop} and setting $t$ to $0$
gives
\[
C'_{X_\Gamma}(0)=C'_{X_{\Gamma\smallsetminus e}\cap X_{\Gamma/e}}(0)
+C_{X_{\Gamma\smallsetminus e}}(0)-C'_{X_{\Gamma\smallsetminus e}}(0)\quad.
\]
The value of the derivative at $0$ equals the Euler characteristic of the 
complement (\cite{feynman}, Proposition~3.1); and as $\Gamma\smallsetminus e$ 
is not a forest, then $C_{X_{\Gamma\smallsetminus e}}(0)=0$. In this case,
\[
n-\chi(X_\Gamma)=n-1-\chi(X_{\Gamma\smallsetminus e}\cap X_{\Gamma/e})
-n+1+\chi(X_{\Gamma\smallsetminus e})\quad,
\]
i.e.,
\[
\chi(X_\Gamma)=n+\chi(X_{\Gamma\smallsetminus e}\cap X_{\Gamma/e})
-\chi(X_{\Gamma\smallsetminus e})\quad.
\]
Remarkably, this formula holds as soon as $e$ is a regular edge on $\Gamma$,
as verified in \cite{delecon}, (3.20). In fact ({\em loc.~cit.,\/} Theorem~3.8) this
formula follows from an analogous formula at the level of Grothendieck classes
which holds if $e$ is a regular edge on~$\Gamma$, regardless of whether
conditions~I and~II are verified. We find it very mysterious that
these conditions should affect the CSM classes
involved in the deletion-contraction formula, but not affect their zero-dimensional 
terms.
\qede
\end{remark}

\subsection{}
While the argument proving the deletion-contraction formula requires the technical
conditions~I and~II to hold for $(\Gamma,e)$, there could be a legitimate doubt
that the formula itself may hold for more general edges; after all, deletion-contraction
for Grothendieck classes does hold in a more general situation (cf.~Remark~\ref{myst}).
The example that follows shows that the formula does not necessarily hold if the
second condition fails.

\begin{example}\label{IIfails}
Condition~II fails for the graph
\begin{center}
\includegraphics[scale=.5]{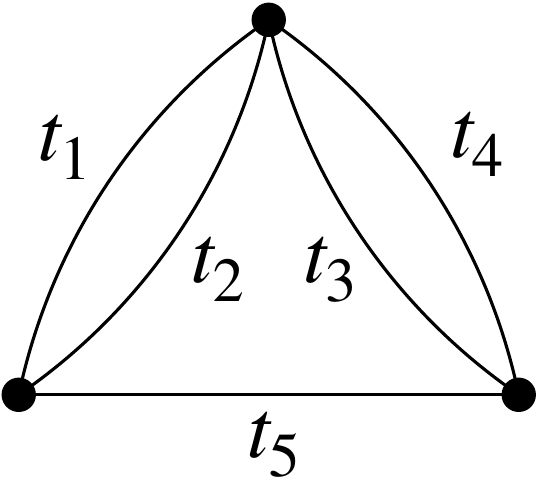}
\end{center}
with respect to the bottom edge $e$ (Example~\ref{cIIexample}; also cf.~\S\ref{disjoinable}). 
Labeling the edges by coordinates as indicated, the graph polynomial is
\[
\Psi_\Gamma= \quad
t_5 (t_1+t_2)(t_3+t_4) + (t_1 t_2 t_3 + t_1 t_2 t_4 + t_1 t_3 t_4 + t_2 t_3 t_4)\quad.
\]
The corresponding hypersurface $X_\Gamma$ is singular along two nonsingular
conics meeting at the point $p=(0 : 0 : 0 : 0 : 1)$. The blow-up $\Til X_\Gamma$
is singular along the proper transforms of these two conics, and along a curve
contained in the exceptional divisor. As the exceptional divisor contains a component
of $\partial \Til X_\Gamma$, it is clear that condition~II is not satisfied, 
cf.~Claim~\ref{condIIex}. It is equally straightforward to verify that condition~I {\em is}
satisfied in this case.

Since $X_\Gamma$ is 
nonsingular in codimension~$1$, its codimension-$0$ and $1$ terms must 
agree with the Chern class of its virtual tangent bundle, i.e., with the class for
a nonsingular hypersurface of degree~$3$ in $\Pbb^4$:
\[
\csm(X_\Gamma)=3[\Pbb^3]+6[\Pbb^2]+\dots\quad.
\]
This observation suffices to determine
\[
C_{X_\Gamma}(t)=t^5+2\,t^4+\underline{4}\,t^3+\text{l.o.t.}\quad.
\]
Deletion and contraction:
\begin{center}
\includegraphics[scale=.5]{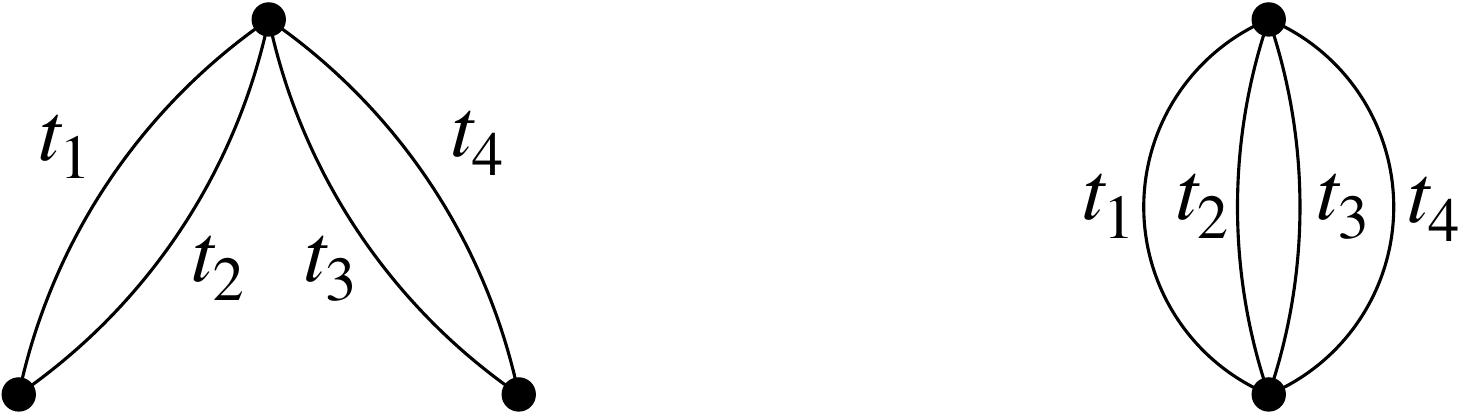}
\end{center}
have polynomials
\begin{align*}
\Psi_{\Gamma\smallsetminus e} &=\quad (t_1+t_2)(t_3+t_4) \\
\Psi_{\Gamma/ e} &=\quad t_1 t_2 t_3 + t_1 t_2 t_4 + t_1 t_3 t_4 + t_2 t_3 t_4\quad.
\end{align*}
We have
\[
C_{X_{\Gamma\smallsetminus e}}(t) = t^2(t+1)^2\quad:
\]
this is easy to obtain directly, as $X_{\Gamma\smallsetminus e}$ consists of the union
of two planes in $\Pbb^3$; and it also follows from the formularium in Proposition~3.1 
of~\cite{feynman}, and Theorem~3.6 of~{\em loc.~cit.\/}
As for $X_{\Gamma\smallsetminus e}\cap X_{\Gamma/e}$, this is easily checked
to consist of three lines in $\Pbb^3$, meeting at two points.
\begin{center}
\includegraphics[scale=.5]{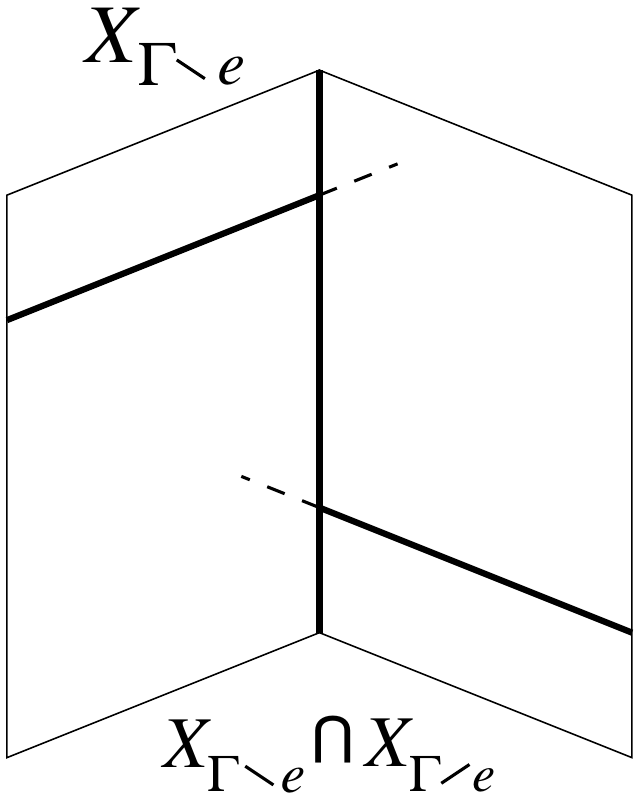}
\end{center}
It follows that $\csm(X_{\Gamma\smallsetminus e}\cap X_{\Gamma/e})=3([\Pbb^1]
+2[\Pbb^0])-2[\Pbb^0]=3[\Pbb^1] +4[\Pbb^0]$, and hence
\[
C_{X_{\Gamma\smallsetminus e}\cap X_{\Gamma/e}}(t) = t^4+4\,t^3+3\,t^2\quad.
\]
Thus, we see that
\[
C_{X_{\Gamma\smallsetminus e}\cap X_{\Gamma/e}}(t)
+(t-1) C_{X_{\Gamma\smallsetminus e}}(t) = t^5+2\, t^4 + \underline{3}\, t^3+2\, t^2
\ne C_{X_\Gamma}(t)\quad,
\]
verifying that the formula in Theorem~\ref{delconthm} need not hold if condition~II
fails.

A more thorough analysis shows that $C_{X_\Gamma}(t)=t^5+2\, t^4 + 4\, t^3+2\, t^2$, 
so that the underlined coefficient is the only discrepancy.
In fact, this may be checked by using Theorem~\ref{delconthm}, using deletion-contraction
with respect to a diagonal edge $e'$; both conditions~I and~II hold in this case by 
Lemma~\ref{mainlemma}. Deletion and contraction are:
\begin{center}
\includegraphics[scale=.5]{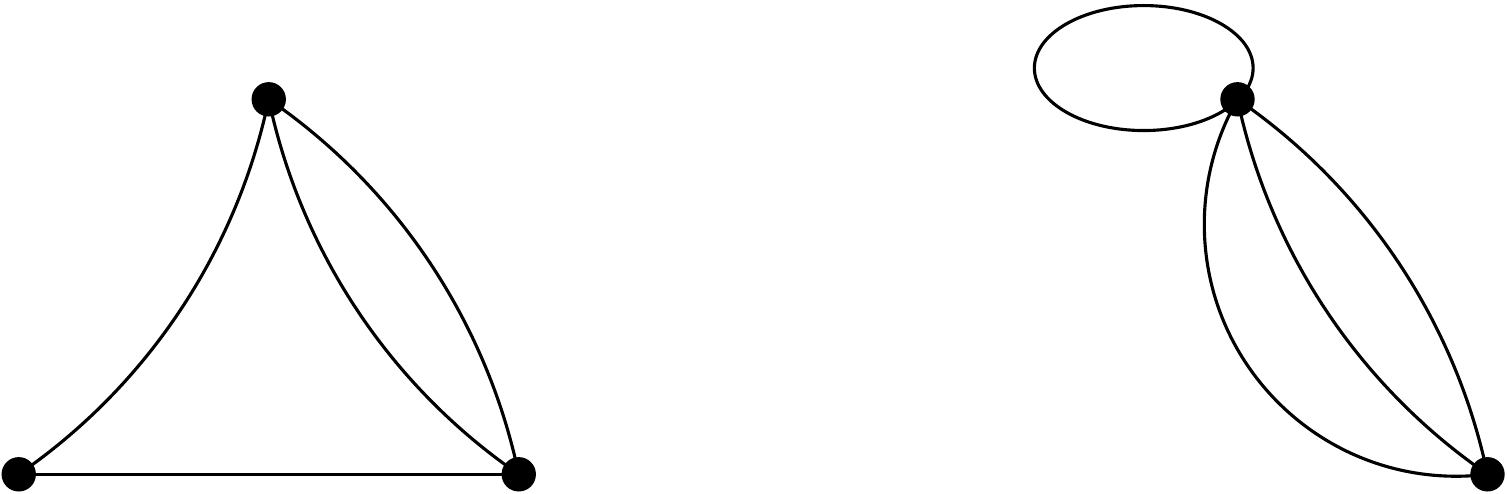}
\end{center}
The reader can check that the graph hypersurface for the deletion is a quadric
cone in $\Pbb^3$, and hence $\csm(X_{\Gamma\smallsetminus e'})
=2[\Pbb^2]+4[\Pbb^1]+3[\Pbb^0]$. The intersection $X_{\Gamma\smallsetminus e'}
\cap X_{\Gamma/e'}$ consists of the union of a nonsingular conic and a line
meeting at a point, hence $\csm(X_{\Gamma\smallsetminus e'}
\cap X_{\Gamma/e'})=3[\Pbb^1]+3[\Pbb^0]$. This yields
\[
C_{X_{\Gamma\smallsetminus e'}}(t)=t^4+2\,t^3+2\,t^2+t
\quad,\quad
C_{X_{\Gamma\smallsetminus e'}\cap X_{\Gamma/e'}}(t)=t^4+4\,t^3+3\,t^2+t
\]
and hence
\[
C_{X_\Gamma}(t)=(t^4+4\,t^3+3\,t^2+t)+(t-1)(t^4+2\,t^3+2\,t^2+t)=
t^5+2\, t^4 + 4\, t^3+2\, t^2
\]
as claimed, by Theorem~\ref{delconthm}.
\qede
\end{example}

As pointed out in Remark~\ref{InotII}, we do not know an example for which the first 
condition fails while the second one holds. We list the relevant classes in an example
for which both conditions fail.

\begin{example}\label{Ifails}
Condition~I fails for the graph
\begin{center}
\includegraphics[scale=.5]{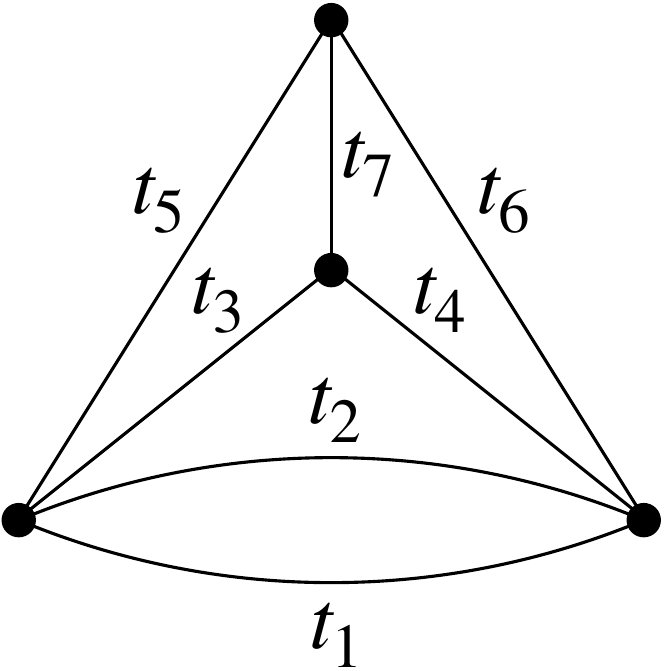}
\end{center}
with respect to the vertical edge $e$ (Example~\ref{cIexample}). Labeling the edges 
by coordinates as indicated, the graph polynomial is
\begin{multline*}
\Psi_\Gamma= \quad
t_7 t_6 t_4 t_2 + t_7 t_6 t_4 t_1 + t_7 t_6 t_3 t_2 + t_7 t_6 t_3 t_1 + t_7 t_6 t_1 t_2 
+ t_7 t_5 t_4 t_2 + t_7 t_5 t_4 t_1 + t_7 t_5 t_3 t_2 \\
+ t_7 t_5 t_3 t_1 + t_7 t_5 t_1 t_2 + t_7 t_4 t_1 t_2 + t_7 t_1 t_2 t_3 + t_6 t_5 t_4 t_2 
+ t_6 t_5 t_4 t_1 + t_6 t_5 t_3 t_2 + t_6 t_5 t_3 t_1\\ 
+ t_6 t_5 t_1 t_2 + t_6 t_4 t_3 t_2 + t_6 t_4 t_3 t_1 + t_6 t_1 t_2 t_3 + t_5 t_4 t_3 t_2 
+ t_5 t_4 t_3 t_1 + t_5 t_4 t_1 t_2 + t_1 t_2 t_3 t_4\quad,
\end{multline*}
and the corresponding $X_\Gamma$ is a hypersurface of degree~$4$ in $\Pbb^6$.
The computation of the terms needed to verify the formula in Theorem~\ref{delconthm}
for this case is more involved than in Example~\ref{IIfails}, and we omit the details.
(The Macaulay2 procedure accompanying \cite{MR1956868} was used for these 
computations.) We obtain:
\[
\csm(X_\Gamma)=4[\Pbb^5]+12[\Pbb^4]+26[\Pbb^3]+29[\Pbb^2]+21[\Pbb^1]
+7[\Pbb^0]\quad,
\]
yielding
\[
C_{X_\Gamma}(t)=t^7+3\,t^6+9\,t^5+9\,t^4+\underline{6}\,t^3\quad.
\]
As for deletion and contraction:
\begin{center}
\includegraphics[scale=.4]{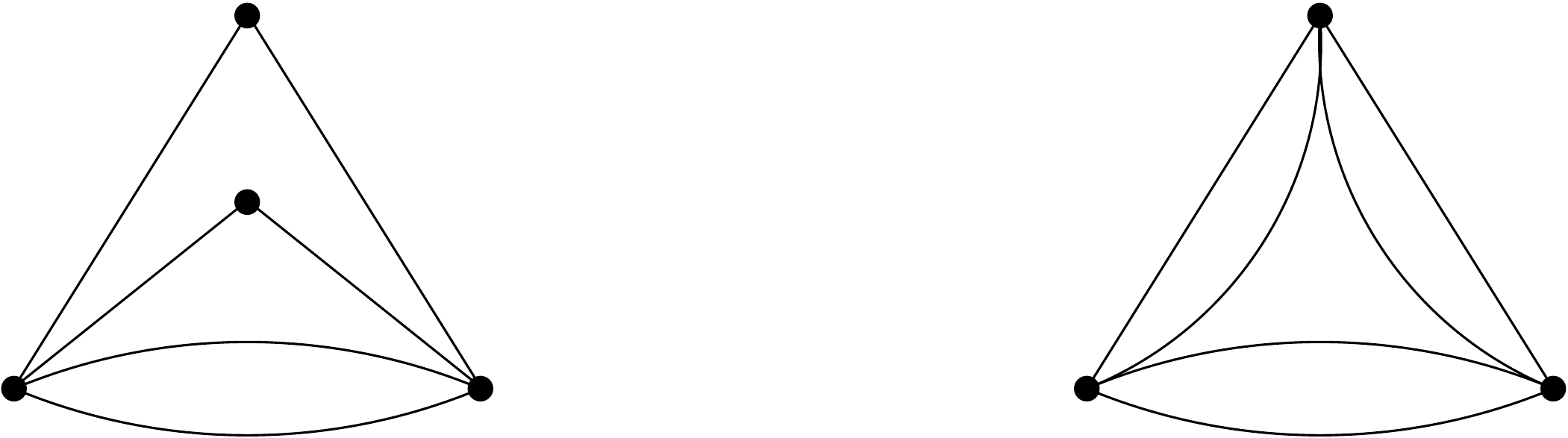}
\end{center}
\begin{gather*}
\csm(X_{\Gamma\smallsetminus e}) = 3[\Pbb^4]+9[\Pbb^3]+14[\Pbb^2]
+14[\Pbb^1]+7[\Pbb^0] \\
\csm(X_{\Gamma\smallsetminus e}\cap X_{\Gamma/e}) 
= 6[\Pbb^3]+10[\Pbb^2]+13[\Pbb^1]+7[\Pbb^0]
\end{gather*}
from which
\begin{align*}
C_{X_{\Gamma\smallsetminus e}}(t) &=t^6+3\,t^5+6\,t^4+6\,t^3+t^2-t \\
C_{X_{\Gamma\smallsetminus e}\cap X_{\Gamma/e}}(t) &=t^6 + 6\, t^5 +9\, t^4
+10\,t^3+2\,t^2-t
\end{align*}
and therefore
\[
C_{X_{\Gamma\smallsetminus e}\cap X_{\Gamma/e}}(t)
+(t-1)C_{X_{\Gamma\smallsetminus e}}(t)
=t^7+3\, t^6+9\, t^5+9\, t^4+\underline{5}\, t^3\quad.
\]
This differs from $C(X_\Gamma)$ by $t^3$, and shows that the formula
in Theorem~\ref{delconthm} need not hold if condition~I is not satisfied.
(But note that Condition~II also fails in this example.)
\qede
\end{example}

\subsection{}\label{trivca}
If $e$ is {\em not\/} a regular edge of $\Gamma$, then the corresponding 
deletion-contraction formulas are trivial consequences
of properties of $C_{X_\Gamma}(t)$ listed in \cite{feynman}, \S3 (especially
Proposition~3.1). 

\begin{itemize}
\item If $e$ is a bridge in $\Gamma$, then
\[
C_{X_\Gamma}(t)=(t+1) C_{X_{\Gamma\smallsetminus e}}(t)\quad.
\]
\item If $e$ is a looping edge in $\Gamma$, then
\[
C_{X_\Gamma}(t)=t\, C_{X_{\Gamma\smallsetminus e}}(t)\quad.
\]
\end{itemize}

If $e$ is a looping edge, then $\Gamma\smallsetminus e=\Gamma/e$, 
and hence $X_{\Gamma\smallsetminus e}\cap X_{\Gamma/e}=
X_{\Gamma\smallsetminus e}$. Thus, the formula given above matches 
the formula obtained by applying Theorem~\ref{delconthm}. 
The formula in Theorem~\ref{delconthm} is also trivially satisfied if
$\Gamma\smallsetminus e$ is a forest.

The formula for bridges has a transparent geometric explanation. If $e$
is a bridge, then $\Psi_\Gamma=\Psi_{\Gamma\smallsetminus e}$; thus 
$X_\Gamma$ is a cone over $X_{\Gamma\smallsetminus e}$, and the
formula given above follows easily from this fact.


\section{Multiple-edge formulas}\label{multedges}

\subsection{}
Deletion-contraction formulas may be used to obtain formulas for the operation of
`multiplying edges', i.e., inserting edges parallel to a given edge of a graph. 
This is carried out in \cite{delecon}, \S5, for the case of the Grothendieck class.
As in Theorem~\ref{delconthm}, the deletion-contraction formula involves a 
`non-combinatorial' term 
(the Grothendieck class of the intersection $X_{\Gamma\smallsetminus e}\cap 
X_{\Gamma/e}$). By virtue of a propitious cancellation, the resulting formula for
doubling an edge only relies on combinatorial data:
\[
\Ubb(\Gamma_{2e})=(\Tbb-1)\,\Ubb(\Gamma)+\Tbb\, \Ubb(\Gamma\smallsetminus e)
+(\Tbb+1)\, \Ubb(\Gamma/e)\quad.
\]
(\cite{delecon}, Proposition~5.2), provided $e$ is a regular edge on $\Gamma$,
and denoting by $\Gamma_{2e}$ the graph obtained by adding an edge parallel
to $e$.

In this section we show that a similar situation occurs for the CSM invariant.
Again, the deletion-contraction formula (Theorem~\ref{delconthm}) involves a
summand which we are not able to interpret directly in combinatorial terms
(that is, $C_{X_{\Gamma\smallsetminus e}\cap X_{\Gamma/e}}(t)$); and again
a fortunate cancellation leads to a purely combinatorial formula for doubling edges.
We will prove:

\begin{theorem}\label{doubled}
Let $\Gamma$ be a graph, and let $e$ be a regular edge of $\Gamma$ such that
$(\Gamma,e)$ satisfies conditions~I and~II of \S\ref{condis}. Denote by $\Gamma_{2e}$
the graph obtained by doubling the edge $e$. Then
\[
C_{\Gamma_{2e}}(t)=(2t-1)\, C_\Gamma(t)
-t(t-1)\, C_{\Gamma\smallsetminus e}(t)+C_{\Gamma/e}(t)
\]
\end{theorem}
\noindent (Here we write $C_\Gamma(t)$ for $C_{X_\Gamma}(t)$, etc.)

By Lemma~\ref{mainlemma}, this formula applies in particular if $e$ is already multiple 
in $\Gamma$. This fact will be used in the proof of Theorem~\ref{doubled}, and will 
lead to multiple-edge formulas in \S\ref{multiples}.

The formula also holds if $e$ is a looping edge, interpreting $C_{\Gamma/e}(t)$
to be $0$ in this case (cf.~\S\ref{trivca}).

\subsection{}
{\em Proof of Theorem~\ref{doubled}.\/}
By Theorem~\ref{delconthm}, under the hypothesis of the theorem, we have
\begin{align*}
C_{X_\Gamma}(t) &= C_{X_{\Gamma\smallsetminus e}\cap X_{\Gamma/e}}(t)
+(t-1)\, C_{X_{\Gamma\smallsetminus e}}(t)\quad, \\
C_{X_{\Gamma_{2e}}}(t) &= C_{X_\Gamma\cap X_{\Gamma_{2e}/e'}}(t)
+(t-1)\, C_{X_\Gamma}(t)\quad,
\end{align*}
where $e'$ denotes the edge parallel to $e$ in 
$\Gamma_{2e}$. Indeed, the first formula holds as $(\Gamma,e)$ satisfies
conditions~I and~II by hypothesis, and the second formula holds since
$(\Gamma_{2e},e')$ satisfies conditions~I and~II by Lemma~\ref{mainlemma};
note that $\Gamma_{2e}\smallsetminus e'=\Gamma$. The theorem will be
obtained by comparing the two intersections $X_{\Gamma\smallsetminus e}
\cap X_{\Gamma/e}\subseteq \Pbb^{n-2}$ and 
$X_\Gamma\cap X_{\Gamma_{2e}/e'}\subseteq \Pbb^{n-1}$
(where $n=$ number of edges of $\Gamma$).

Let $t_e$, $t_{e'}$ be the variables corresponding to the two parallel edges $e$, $e'$
in $\Gamma_{2e}$, and let $t_1,\dots,t_{n-1}$ be the variables corresponding to
the other edges. We have
\[
\Psi_{\Gamma_{2e}}=t_{e'} \Psi_\Gamma+ \Psi_{\Gamma_{2e}/ e'}\quad.
\]
The graph $\Gamma_{2e}/ e'$ may be obtained by attaching a 
looping edge marked $e$ at the vertex obtained by contracting $e$ in $\Gamma$:
\begin{center}
\includegraphics[scale=.5]{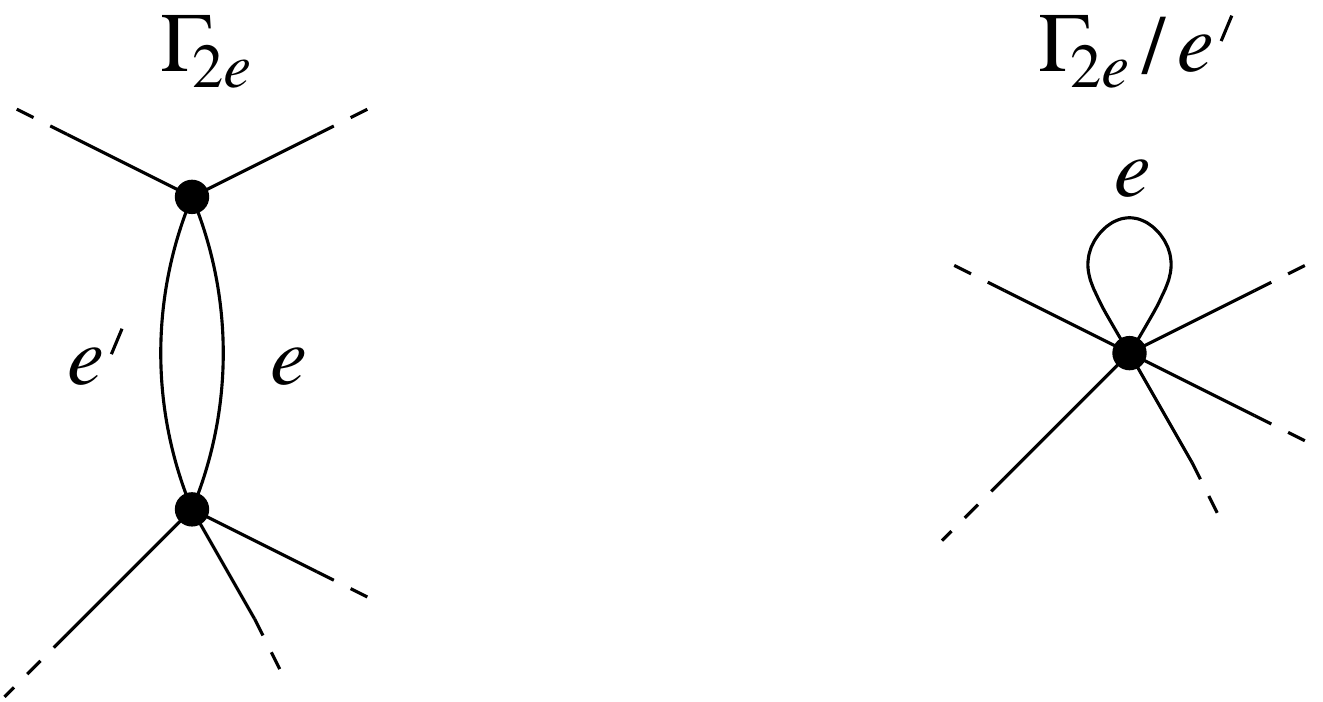}
\end{center}
As a consequence, 
\[
\Psi_{X_{\Gamma_{2e}/e'}} = t_e \Psi_{X_{\Gamma/e}}\quad,
\]
and $X_{\Gamma\smallsetminus e} \cap X_{\Gamma/e}$,  
$X_\Gamma\cap X_{\Gamma_{2e}/e'}$ have ideals
\[
\left(\Psi_{\Gamma\smallsetminus e} , \Psi_{\Gamma/e}\right)\quad,\quad
\left(\Psi_{\Gamma} , t_e \Psi_{\Gamma/e}\right)
\]
respectively. The first should be viewed as an ideal in $k[t_1,\dots,t_{n-1}]$, 
and the second as an ideal in $k[t_1,\dots,t_{n-1},t_e]$. Denoting by $V(f_1,f_2,\dots)$
the locus $f_1=f_2=\dots=0$, we have
\begin{align*}
V(\Psi_{\Gamma} , t_e \Psi_{\Gamma/e})
&=V(\Psi_{\Gamma} , t_e)\cup V(\Psi_{\Gamma} , \Psi_{\Gamma/e})
=V(\Psi_{\Gamma/e} , t_e)\cup V(t_e\Psi_{\Gamma\smallsetminus e} , \Psi_{\Gamma/e})
\intertext{
using the fact that $\Psi_\Gamma=t_e \Psi_{\Gamma\smallsetminus e}+\Psi_{\Gamma/e}$,
}
&=V(\Psi_{\Gamma/e} , t_e)\cup V(\Psi_{\Gamma\smallsetminus e} , \Psi_{\Gamma/e})
\quad.
\end{align*}
This shows that $X_\Gamma\cap X_{\Gamma_{2e}/e'}$ is the union of a copy of
$X_{\Gamma/e}\subseteq \Pbb^{n-2}$ and a cone in~$\Pbb^{n-1}$ over
$X_{\Gamma\smallsetminus e} \cap X_{\Gamma/e}\subseteq \Pbb^{n-2}$.
The intersection of these two loci is $V(\Psi_{\Gamma\smallsetminus e},
\Psi_{\Gamma/e},t_e)$, that is, a copy of 
$X_{\Gamma\smallsetminus e} \cap X_{\Gamma/e}\subseteq \Pbb^{n-2}$.

The invariant $C_-(t)$ satisfies an inclusion-exclusion property because so does the 
Chern-Schwartz-MacPherson class, and its behavior with respect to taking a cone 
amounts to multiplication by $(t+1)$: this follows easily from the formula for CSM 
classes of cones, see Proposition~5.2 in \cite{MR2504753}. Therefore,
\begin{align*}
C_{X_\Gamma\cap X_{\Gamma_{2e}/e'}}(t)
&=C_{X_{\Gamma/e}}(t)+(t+1)\, C_{X_{\Gamma\smallsetminus e} \cap X_{\Gamma/e}}(t)
-C_{X_{\Gamma\smallsetminus e} \cap X_{\Gamma/e}}(t) \\
&=C_{X_{\Gamma/e}}(t)+t\, C_{X_{\Gamma\smallsetminus e} \cap X_{\Gamma/e}}(t)
\quad.
\end{align*}
Using this fact together with the two formulas given at the beginning of the proof,
we get
\begin{align*}
C_{X_{\Gamma_{2e}}}(t)-(t-1)\, C_{X_\Gamma}(t)
&=C_{X_{\Gamma/e}}(t)+t\, C_{X_{\Gamma\smallsetminus e} \cap X_{\Gamma/e}}(t)\\
&=C_{X_{\Gamma/e}}(t)
+t\, \left(C_{X_\Gamma}(t)-(t-1)\, C_{X_{\Gamma\smallsetminus e}}(t)\right)\quad,
\end{align*}
which yields immediately the stated formula.
\qed

\subsection{}\label{multiples}
By Lemma~\ref{mainlemma}, a multiple regular edge satisfies both conditions~I and~II,
and hence the doubling edge formula of Theorem~\ref{doubled} applies to it.
For a regular edge $e$ on a graph $\Gamma$, and $m\ge 1$, denote by 
$\Gamma_{me}$ the graph obtained by replacing $e$ with $m$ edges connecting 
the same vertices.
In particular, $\Gamma=\Gamma_{e}$.

\begin{theorem}\label{goodform}
Let $e$ be a regular edge of $\Gamma$. With notation as above,
\[
\sum_{m\ge 0} C_{\Gamma_{(m+1)e}}(t)\,\frac{s^m}{m!}
=e^{ts}\left(K(t)\, C_{\Gamma}(t)-K'(t)\, C_{\Gamma_{2e}}(t)+\frac{K''(t)}2\,
C_{\Gamma_{3e}}(t)\right)\quad,
\]
where
\[
K(t)=t^2 e^{-s} + (t-1)(ts-t-1)\quad.
\]
Equivalently,
\begin{align*}
C_{\Gamma_{(m+1)e}}(t)
&=\left(t^2\, C_{\Gamma}(t)-2t\, C_{\Gamma_{2e}}(t)+C_{\Gamma_{3e}}(t)\right) (t-1)^{m-1}\\
&-\left((t^2-1)\, C_{\Gamma}(t)-2t\, C_{\Gamma_{2e}}(t)+C_{\Gamma_{3e}}(t)\right) t^{m-1}\\
&+\left((t^2-t)\, C_{\Gamma}(t)-(2t-1)\, C_{\Gamma_{2e}}(t)+C_{\Gamma_{3e}}(t)\right) 
(m-1)\,t^{m-2}\quad.
\end{align*}
\end{theorem}

The fact that the coefficient of $C_\Gamma(t)$ determines the others
by taking derivatives is a consequence of the analogous feature displayed by
the coefficients of the basic recursion for $C_{\Gamma_{me}}$:

\begin{lemma}\label{recurs}
Let $e$ be a regular edge of $\Gamma$. Then for $m\ge 1$
\[
C_{\Gamma_{(m+3)e}}(t)=
(3t-1)\, C_{\Gamma_{(m+2)e}}(t) - (3t^2-2t)\, C_{\Gamma_{(m+1)e}}(t) 
+ (t^3-t^2)\, C_{\Gamma_{me}}(t)\quad.
\]
\end{lemma}

This feature reflects the fact that the characteristic polynomial
for the recursion is a function of $t-x$:
\[
x^3-(3t-1)\,x^2+t(3t-2)\,x-t^2(t-1)
=-(t-x)^2(t-x-1)\quad.
\]
This fact is intriguing, and we do not have a conceptual explanation for it.
(Note that the corresponding fact is not verified for the recursion
computing the Grothendieck class. Using (5.8) in \cite{delecon}, it is immediate
to show that
\[
\Ubb(\Gamma_{(m+3)e}) = (2\Tbb -1)\, \Ubb(\Gamma_{(m+2)e})
-\Tbb\,(\Tbb-2)\, \Ubb(\Gamma_{(m+1)e})
-\Tbb^2\, \Ubb(\Gamma_{me})\quad,
\]
with notation as in loc.~cit., if $e$ is regular in $\Gamma$. The characteristic 
polynomial factors as $(x+1)\, (\Tbb-x)^2$.)
Lemma~\ref{recurs} is an immediate consequence of Theorem~\ref{doubled}:

\begin{proof}
If $e$ is a regular edge, then $e$ satisfies conditions~I and~II in $\Gamma_{ne}$
for all $n\ge 2$, by Lemma~\ref{mainlemma}. Apply Theorem~\ref{doubled} to
obtain
\begin{align*}
C_{\Gamma_{(m+2)e}}(t) &=
(2t-1)\, C_{\Gamma_{(m+1)e}}(t) -t(t-1)\, C_{\Gamma_{me}}(t) + t^m\, C_{\Gamma/e}(t) \\
C_{\Gamma_{(m+3)e}}(t) &=
(2t-1)\, C_{\Gamma_{(m+2)e}}(t) -t(t-1)\, C_{\Gamma_{(m+1)e}}(t) 
+ t^{m+1} \,C_{\Gamma/e}(t)
\end{align*}
for $m\ge 1$.
The stated recursion follows immediately, by eliminating $C_{\Gamma/e}(t)$ from 
these expressions.
\end{proof}

Theorem~\ref{goodform} follows directly from this lemma, by standard methods.

\begin{remark}\label{multicor}
Theorem~\ref{goodform} may be rewritten in the following fashion: for
all regular edges $e$ of $\Gamma$,
\begin{multline*}
\sum_{m\ge 0} C_{\Gamma_{(m+1)e}}(t)\, \frac {s^m}{m!}
=\left(e^{ts}-e^{(t-1)s}\right) C_{\Gamma_{2e}}(t)\\
-\left((t-1)\, e^{ts}-t\, e^{(t-1)s}\right) C_{\Gamma}(t)
+t\left((s-1)\, e^{ts}+e^{(t-1)s}\right) C_{\Gamma/e}(t)\quad.
\end{multline*}
That is, $C_{\Gamma_{me}}(t)$ is given by the expression
\[
\left(C_{\Gamma_{2e}}(t)-t\,C_{\Gamma}(t)-t\,C_{\Gamma/e}(t)\right)
\left(t^{m-1}-(t-1)^{m-1}\right)
+\left(C_{\Gamma}(t)+(m-1)\, C_{\Gamma/e}(t)\right) t^{m-1}
\]
for all $m\ge 1$.
These expressions are simpler to apply than Theorem~\ref{goodform}, 
since $C_{\Gamma/e}(t)$ is usually more immediately accessible than 
$C_{\Gamma_{3e}}$.
\qede
\end{remark}

\begin{example}\label{bananaex}
The $n$-edge {\em banana graph\/} consists of $n$ edges connecting two 
distinct vertices.
\begin{center}
\includegraphics[scale=.5]{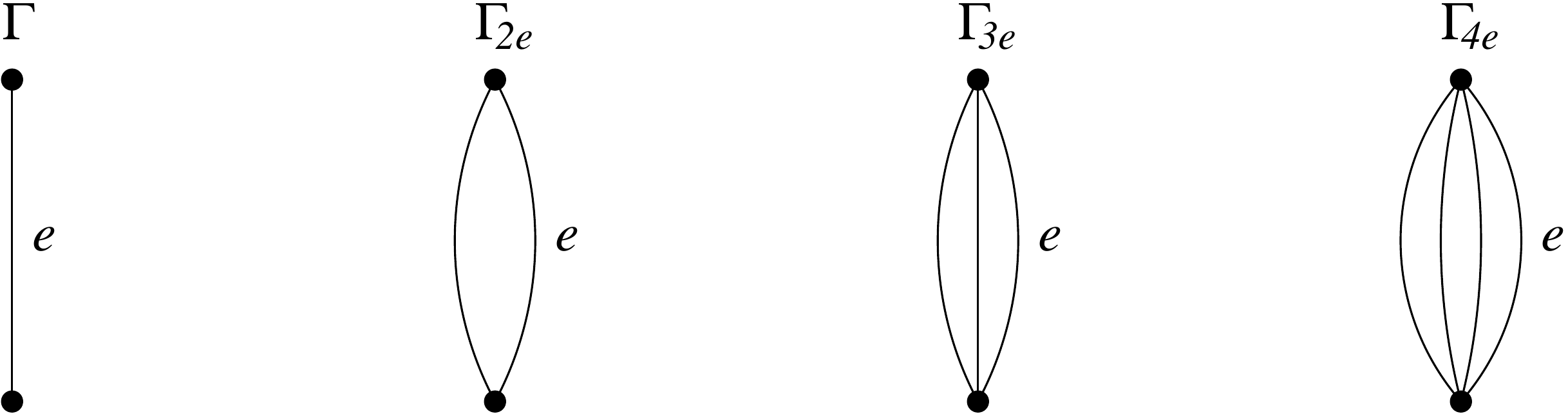}
\end{center}
We have
\[
C_\Gamma(t)=t+1\quad,\quad
C_{\Gamma_{2e}}(t)=t\,(t+1)\quad,\quad
C_{\Gamma/e}(t)=1\quad:
\]
indeed, $\Gamma$ is a single bridge, and $\Gamma_{2e}$ is a $2$-polygon
(see Proposition~3.1 in~\cite{feynman}). The formula in Remark~\ref{multicor} 
yields the following pretty generating function:
\[
\sum_{m\ge 0} C_{\Gamma_{(m+1)e}}(t)\,\frac {s^m}{m!}
=(1+ts)\,e^{ts}+t\,e^{(t-1)s}\quad,
\]
or equivalently
\[
C_{\Gamma_{ne}}(t)=n\, t^{n-1}+t\,(t-1)^{n-1}
\]
for $n\ge 1$. This agrees with the formula given in Example~3.8 of~\cite{feynman},
which was obtained by a very different method. 
\qede
\end{example}

\subsection{}
The coefficients obtained in Remark~\ref{multicor} agree 
with the ones conjectured in \cite{delecon}, \S6.2, with the difference that
that the formulas given here are applied to~$\Gamma_{2e}$ rather than
$\Gamma=\Gamma_{e}$; this accounts for the extra factor of $t$ in 
the last coefficient. The point is that the hypotheses of Theorem~\ref{doubled}
are automatically satisfied for~$\Gamma_{2e}$ since $e$ has parallel edges
in $\Gamma_{2e}$, while they are not necessarily satisfied for~$\Gamma$.
Accordingly, while it is tempting to interpret $\Gamma_{0e}$ as 
$\Gamma\smallsetminus e$, the generating function given in Remark~\ref{multicor} 
cannot be extended in general to provide information about this graph.

\begin{example}\label{doubletriangex}
We verify that Theorem~\ref{doubled} does not necessarily hold
if $e$ does not satisfy the hypotheses presented in \S\ref{condis}. 
For this, we return to the graph of Examples~\ref{cIIexample} and~\ref{IIfails}.
\begin{center}
\includegraphics[scale=.5]{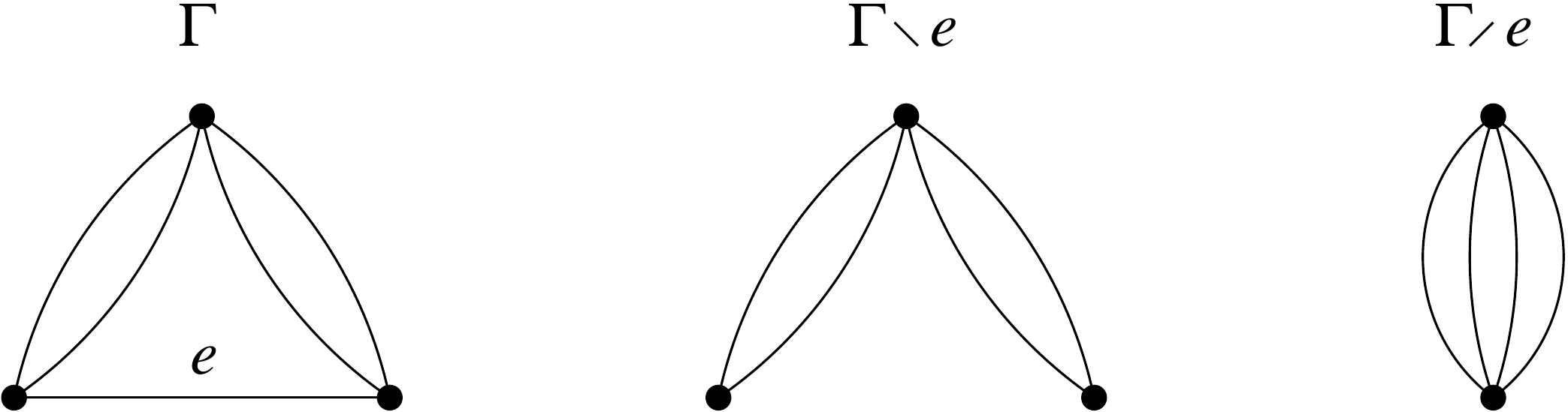}
\end{center}
We have (cf.~Examples~\ref{IIfails} and~\ref{bananaex})
\begin{align*}
C_{\Gamma}(t) &=t^5+2\, t^4+4\, t^3+2\, t^2 \\
C_{\Gamma\smallsetminus e}(t) &=(t+1)^2\, t^2 \\
C_{\Gamma/e}(t) &=4\,t^3+t(t-1)^3\quad;
\end{align*}
the formula in Theorem~\ref{doubled} would give
\begin{equation*}
\tag{*}
(2t-1)\, C_\Gamma(t)-t(t-1)\,C_{\Gamma\smallsetminus e}(t)+C_{\Gamma/e}(t)
=t^6+2\,t^5+\underline{8}\,t^4+2\,t^3+t^2-t
\end{equation*}
and we can verify that this does {\em not\/} equal $C_{\Gamma_{2e}}(t)$.
\begin{center}
\includegraphics[scale=.5]{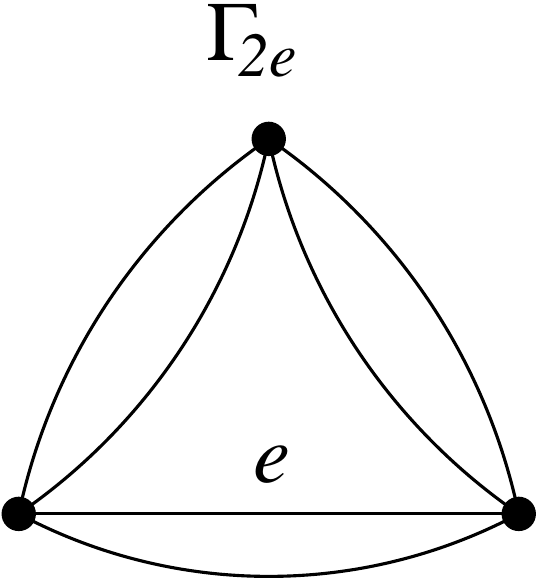}
\end{center}
Indeed, the graph polynomial for $\Gamma_{2e}$ is
\begin{multline*}
\Psi_{\Gamma_{2e}}=t_6t_5t_4t_2+t_6t_5t_4t_1+t_6t_5t_3t_2+t_6t_5t_3t_1
+t_6t_4t_3t_2+t_6t_4t_3t_1\\
+t_6t_4t_1t_2+t_6t_1t_2t_3+t_5t_4t_3t_2+t_5t_4t_3t_1+t_5t_4t_1t_2+t_5t_1t_2t_3
\end{multline*}
(where the pairs of variables $(t_1,t_2)$, $(t_3,t_4)$, $(t_5,t_6)$ correspond to the
three pairs of parallel edges).
Macaulay2 confirms that $X_{\Gamma_{2e}}$ is nonsingular in codimension~$1$,
hence
\[
\csm(X_{\Gamma_{(2)}})=4\, [\Pbb^4]+8\, [\Pbb^3]+\cdots
\]
from which
\[
C_{\Gamma_{2e}}(t)=t^6+2\,t^5+\underline{7}\,t^4+\text{l.o.t.}\quad,
\]
differing from (*).
In fact, a computation using the code from \cite{MR1956868} shows that
\[
C_{\Gamma_{2e}}(t)=t^6+2\,t^5+7\,t^4+2\,t^3+t^2-t\quad;
\]
this differs from (*) by exactly $t^4$. The corresponding CSM class is
\[
4[\Pbb^4]+8[\Pbb^3]+18[\Pbb^2]+14[\Pbb]+7[\Pbb^0]\quad,
\]
correcting the formula given at the end of~\S6 in~\cite{delecon}.
(Incidentally, the fact that in these examples
the discrepancies occurring when conditions~I and~II fail are pure powers of $t$
is also intriguing, and calls for an explanation.) 
By Remark~\ref{multicor}, we also obtain that
\[
C_{\Gamma_{me}}(t)
=(t^2-t+1)^2\,t\, (t-1)^{m-1}
+\left(4t^3+t^2+4t-1
+ (m-1)\,(t^3+t^2+3t-1)\right) t^m
\]
for all $m\ge 1$.
\begin{center}
\includegraphics[scale=.5]{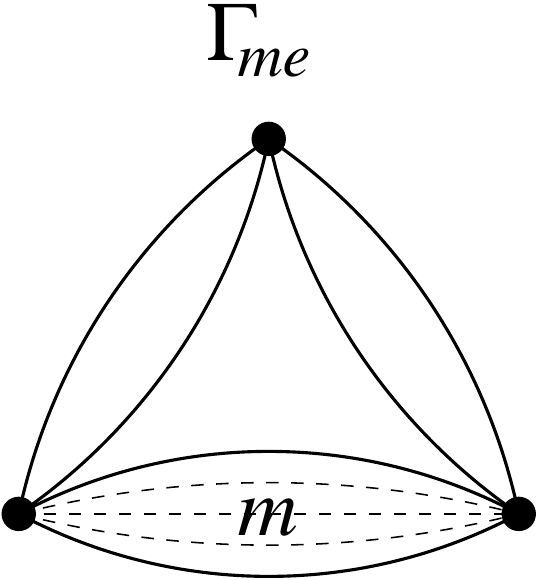}
\end{center}
\end{example}


\section{Alternative approach, via specialization}\label{speci}

\subsection{}\label{fami}
In this section we explain briefly a different approach to the question studied
in this paper, based on the theory of specialization of Chern classes; this theory
is originally due to Verdier (\cite{MR629126}). For simplicity, we now work over $\Cbb$.

Let $\{X_u\}$ be a family of hypersurfaces over a disk; assume that the fibers over
$u\ne 0$ are all isomorphic. Under suitable (and mild) hypotheses, one may define
a specific constructible function $\sigma$ on the central fiber $X=X_0$, with the
property that $\csm(\sigma)$ equals the specialization (in the sense of intersection
theory) of the CSM class of the general fiber. The value $\sigma(p)$ at a point
$p\in X$ equals the Euler characteristic of the intersection
of the $\epsilon$-ball centered at $p$ with nearby fibers $X_u$,
as $\epsilon\to 0$ and $|u|\ll \epsilon$. See \cite{MR629126}, Proposition~4.1.
It may also be computed in terms of an embedded resolution of $X$, 
see~\cite{verdier}.

In the situation we will consider here, the hypersurfaces $X_u$ will be elements
of a pencil in projective space. It is easy to verify that $\sigma(p)=1$ if $X$ is 
nonsingular at $p$, and $\sigma(p)=0$ if $p$ is a point of transversal intersection 
of two nonsingular components of $X$, provided that $p$ does not belong to
the base locus of the pencil. 
Parusi\'nski and Pragacz (\cite{MR2002g:14005}, Proposition 5.1) prove that 
$\sigma(p)=1$ if $p\in X_u\cap X$ ($u\ne 0$) {\em is\/} a point of the base locus, 
if $X_u$ is {\em smooth and transversal to the strata of a fixed Whitney stratification 
of $X$\/} at $p$. This beautiful observation will be used below.

\subsection{}
Returning to graph hypersurfaces, recall that if $e$ is a regular edge of $\Gamma$,
then
\[
\Psi_{\Gamma}=t_e\, \Psi_{\Gamma\smallsetminus e}+\Psi_{\Gamma/e}\quad;
\]
as $\Gamma\smallsetminus e$ is not a forest, both $\Psi_{\Gamma\smallsetminus e}$ 
and $\Psi_{\Gamma/e}$ are polynomials of positive degree.
Thus, we may view $\Psi_{\Gamma}$ as an element of the pencil spanned
by $t_e\, \Psi_{\Gamma\smallsetminus e}$ and $\Psi_{\Gamma/e}$: $u=1$ in
\begin{equation*}
\tag{*}
\Psi_u=t_e\, \Psi_{\Gamma\smallsetminus e}+u\,\Psi_{\Gamma/e}\quad.
\end{equation*}
We note that $X_u:\{\Psi_u=0\}$ is isomorphic to $X_{\Gamma}:\{\Psi_1=0\}$ for 
$u\ne 0$: indeed, with $d=\deg \Psi_{\Gamma}$, 
\begin{align*}
\Psi_u(u^2 t_e, u \mathbf t_{\ne e})
&=u^2\, t_e\Psi_{\Gamma\smallsetminus e}(u \mathbf t_{\ne e})
+u\, \Psi_{\Gamma/e}(u \mathbf t_{\ne e})
=u^{d+1}\, t_e \Psi_{\Gamma\smallsetminus e}(\mathbf t_{\ne e})
+u^{d+1}\, \Psi_{\Gamma/e} (\mathbf t_{\ne e}) \\
&= u^{d+1}\, \Psi_{\Gamma}(t_e,\mathbf t_{\ne e})\quad,
\end{align*}
where $\mathbf t_{\ne e}$ denotes the variables corresponding to edges other than $e$.
View (*) as defining a family as in \S\ref{fami}, with central fiber $X:\{\Psi_0=0\}$.
That is,
\[
X = \cone X_{\Gamma\smallsetminus e} \cup H\quad,
\]
where $n=$ number of edges in the graph $\Gamma$, $H\cong \Pbb^{n-2}$ is the 
hyperplane defined by $t_e=0$, and $\cone X_{\Gamma\smallsetminus e}$ denotes 
the cone over $X_{\Gamma\smallsetminus e}\subseteq H$ with vertex at the point 
$p=(t_e: \mathbf t_{\ne e})=(1: 0:\dots:0)$.
As a consequence of Verdier's theorem (Th\'eor\`eme~5.1 in \cite{MR629126}),
\[
\csm(\one_{X_{\Gamma}})=\csm(\sigma)\quad,
\]
where these classes are taken in $A_*\Pbb^{n-1}$, and $\sigma$ is the specialization
function on $X$ defined in \S\ref{fami}.

\subsection{}\label{condIpstat}
The difficulty with this approach lies in the explicit computation of $\sigma$.
Again we indicate two conditions under which a result may be obtained, matching
the result of the more algebraic approach taken in the rest of the paper.

The first condition is a `set-theoretic' version of condition~I:
\begin{equation*}
\tag{Condition I'} \partial X_{\Gamma\smallsetminus e}\subseteq X_{\Gamma/e}\quad.
\end{equation*}
As pointed out in \S\ref{condI}, Condition~I implies this inclusion at the level of 
schemes; here, we are only requiring it at the level of sets.
If this condition is satisfied, then the value of $\sigma$ may be determined for all points
of $X\smallsetminus X_\Gamma$. This consists of the complements of $X_\Gamma$
in $\cone X_{\Gamma\smallsetminus e}\smallsetminus H$,
in $\cone X_{\Gamma\smallsetminus e}\cap H$, and in
$H\smallsetminus \cone X_{\Gamma\smallsetminus e}$.

\begin{lemma}\label{condIp}
If condition~I' holds, then
\begin{itemize} 
\item For $q\in (\cone X_{\Gamma\smallsetminus e} \smallsetminus H)
\smallsetminus X_{\Gamma}$, $\sigma(q)=1$;
\item For $q\in (\cone X_{\Gamma\smallsetminus e}\cap H)
\smallsetminus X_\Gamma$, $\sigma(q)=0$.
\item For $q\in (H \smallsetminus \cone X_{\Gamma\smallsetminus e})
\smallsetminus X_{\Gamma}$, $\sigma(q)=1$;
\end{itemize}

In fact, $\sigma(q)=1$ for all $q\in H\smallsetminus \cone X_{\Gamma\smallsetminus e}$.
\end{lemma}

\begin{proof}
If condition~I' holds, then the singularities of $X_{\Gamma\smallsetminus e}$ are
contained in $X_{\Gamma/e}$, and it follows easily that the singular locus of the
cone $\cone X_{\Gamma\smallsetminus e}$ is contained in $X_\Gamma$.

On the other hand, note that $\Psi_{\Gamma\smallsetminus e}$ is one of the
partial derivatives of $\Psi_\Gamma$; hence, the singularities of $X_\Gamma$ are
contained in $\cone X_{\Gamma\smallsetminus e}$.

Thus: the first statement holds since $\cone X_{\Gamma\smallsetminus e}$ is
nonsingular at $q$, and $q$ does not belong to the base locus of the pencil.
The third statement likewise holds because $q$ is not in the base locus,
and $H\cong \Pbb^2$ is nonsingular. In the second statement, $q$ is nonsingular on
both $\cone X_{\Gamma\smallsetminus e}$ and $H$, and these hypersuraces
intersect transversally at $p$, so $\sigma(q)=0$ as recalled in \S\ref{fami}.

To prove the last assertion, we have to consider $q\in (H\cap X_\Gamma)
\smallsetminus \cone X_{\Gamma\smallsetminus e}$. (We have already dealt with 
the other points in $H\smallsetminus \cone X_{\Gamma\smallsetminus e}$.) 
At these points, both $X$ and $X_\Gamma$ are nonsingular (by condition~I').
Further, the intersection $H\cap X_\Gamma = X_{\Gamma/e}$ is also nonsingular
at such points, again by condition~I'. Thus the hypotheses of the result of
Parusi\'nski and Pragacz recalled in~\S\ref{fami} are satisfied, and it follows
that $\sigma(q)=1$.
\end{proof}

\subsection{}\label{condIIpstat}
The locus unaccounted for in Lemma~\ref{condIp} is 
$X_\Gamma\cap \cone X_{\Gamma\smallsetminus e}$. These are points in the
base locus which are contained in the component $\cone X_{\Gamma\smallsetminus e}$.
{\em If\/} $X_\Gamma$ were nonsingular and transversal to the strata of 
$\cone X_{\Gamma\smallsetminus e}$ at these points, then (according to the
formula of Parusi\'nski and Pragacz) we would have
$\sigma(q)=1$ for $q\in X_\Gamma\cap \cone X_{\Gamma\smallsetminus e}$;
the deviation of $\sigma$ from $1$ at such points is a measure of 
non-transversality of the two hypersurfaces. Also note that this locus 
equals the cone over $X_{\Gamma\smallsetminus e}\cap X_{\Gamma/e}$
with vertex at $p$.
The following condition should again be interpreted as a subtle 
notion of `transversality' of the intersection of the various loci considered here.
\begin{equation*}
\tag{Condition II'} \text{$\sigma(q)=1$ for $q\in \cone X_{\Gamma\smallsetminus e}
\cap \cone X_{\Gamma/e}$}\quad.
\end{equation*}

\begin{lemma}\label{bothcondp}
Assume both conditions~I' and~II' hold. Then
\[
\sigma=\one_{\cone X_{\Gamma\smallsetminus e}}+\one_H
-\,2\one_{X_{\Gamma\smallsetminus e}}+
\one_{X_{\Gamma\smallsetminus e}\cap X_{\Gamma/e}}
\]
where the latter two mentioned loci are viewed as subsets of $H\cong \Pbb^{n-2}$.
\end{lemma}

\begin{proof}
This follows from Lemma~\ref{condIp}, condition~II', and elementary bookkeeping.
\end{proof}

\subsection{}
Applying Verdier's theorem now recovers the same formula we obtained in 
Theorem~\ref{delconthm}

\begin{theorem}
Let $e$ be a regular edge of $\Gamma$. Assume $(\Gamma,e)$ satisfies 
both conditions~I' and~II' given above. Then
\[
C_{X_\Gamma}(t)=C_{X_{\Gamma\smallsetminus e}\cap X_{\Gamma/e}}(t)
+(t-1)\, C_{X_{\Gamma\smallsetminus e}}(t)\quad.
\]
\end{theorem}

\begin{proof}
Applying Verdier's theorem and additivity of CSM classes, we get
\[
\csm(X) = \csm(\sigma) = 
\csm(\one_{\cone X_{\Gamma\smallsetminus e}})+\csm(\one_H)-
2\,\csm(\one_{X_{\Gamma\smallsetminus e}})
+\csm(\one_{X_{\Gamma\smallsetminus e}\cap X_{\Gamma/e}})\quad.
\]
Recalling $H\cong \Pbb^{n-2}$, and expressing in terms of polynomial Feynman rules,
this gives
\[
C_{X_\Gamma}(t)=(t+1)C_{X_{\Gamma\smallsetminus e}}(t)
-2\, C_{X_{\Gamma\smallsetminus e}}(t)
+C_{X_{\Gamma\smallsetminus e}\cap X_{\Gamma/e}}(t)\quad,
\]
with the stated result. Here we also used the fact that 
$C_{\cone X_{\Gamma\smallsetminus e}}(t)=(t+1)\, C_{X_{\Gamma\smallsetminus e}}$,
an easy consequence of Proposition~5.2 in \cite{MR2504753}.
\end{proof}

\subsection{}
In conclusion, we have verified that the deletion-contraction formula obtained under
the `algebraic' conditions~I and~II described in \S\ref{condis} also holds under
the conditions~I' and~II' given in \S\S\ref{condIpstat}-\ref{condIIpstat}.

These latter conditions have a more `geometric' flavor: condition~I' is a set-theoretic 
statement, and condition~II' amounts to a statement on the Euler characteristics of 
intersections of an $\epsilon$-ball with nearby fibers in a fibration naturally associated 
with the pair $(\Gamma,e)$. In this sense they are perhaps easier to appreciate, although
in practice the algebraic counterparts I and~II are more readily verifiable in given
cases, by means of tools such as \cite{M2}. It would be interesting to relate these
two sets of conditions more precisely: does $(\Gamma,e)$ satisfy I and~II if and
only if it satisfies I' and~II'? Do these conditions admit a transparent combinatorial 
interpretation?



\end{document}